\documentclass{article}
\usepackage{amsmath, amssymb}
\usepackage[utf8]{inputenc}
\usepackage{graphicx} 
\usepackage{hyperref}
\usepackage{color}
\usepackage{framed}
\usepackage{bbold}
\usepackage{mathtools}

\usepackage{algorithm}
\usepackage{algpseudocode}

\usepackage{microtype}

\usepackage[toc,page]{appendix} 

\usepackage{amsthm}  
\newtheorem*{remark*}{Remark}
\newtheorem{remark}{Remark}
\newtheorem{example}{Example}
\newtheorem{proposition}{Proposition}
\newtheorem*{proposition*}{Proposition}
\newtheorem{lemma}{Lemma}
\newtheorem{theorem}{Theorem}
\newtheorem*{theorem*}{Theorem}
\newtheorem{definition}{Definition}

\newcommand{\norm}[1]{\left\lVert #1 \right\rVert}
\newcommand{\inner}[2]{\left\langle #1, #2 \right\rangle}

\usepackage{float}
\usepackage{caption}
\usepackage{subcaption}
\usepackage[numbers]{natbib}
\usepackage{graphicx} 
\usepackage[a4paper, total={6in, 8.5in}]{geometry} 
\usepackage[utf8]{inputenc} 
\usepackage[T1]{fontenc}
\usepackage{lmodern}
\usepackage{titlesec}
\usepackage{etoolbox}

\usepackage{natbib}
\usepackage{url}

\title{Reversible Stratonovich perturbations for accelerating convergence to equilibrium of Langevin dynamics}
\author{ Guillaume Louis\footnotemark[1], Grigorios A. Pavliotis\footnotemark[2], and Gilles Vilmart\footnotemark[1]}
\date{\today}

\begin{document}
\maketitle
\footnotetext[1]{Section de Math\'{e}matiques, Universit\'{e} de Gen\`{e}ve, CP 64, 1211 Gen\`{e}ve 4, Switzerland. guillaume.louis@unige.ch gilles.vilmart@unige.ch}
\footnotetext[2]{Department of Mathematics, Imperial College, 736a, Huxley Building, South Kensington Campus, United Kingdom. g.pavliotis@imperial.ac.uk}

\begin{abstract}
In this paper, we study the problem of sampling from a probability distribution using Langevin-type dynamics. We consider overdamped Langevin dynamics perturbed by adding a Stratonovich perturbation that preserves reversibility. We prove an optimal scaling for the Stratonovich perturbation applied to Gaussian target distributions and derive an algorithm to construct such perturbation. Our theoretical results are supplemented by numerical experiments in which we compare our proposed method with nonreversible samplers and overdamped Langevin dynamics. Our numerical experiments demonstrate the efficiency of the approach beyond the quadratic case.
\end{abstract}



\section{Introduction}

Sampling from a probability measure with density \(\pi(\mathbf{x}) \propto e^{-V(\mathbf{x})}\), appears in numerous applications such as molecular dynamics \cite{MolDynLeimkuhler2015}, computational statistics~\cite{CRSW2013}, Bayesian statistical inverse problems \cite{HUQDashti2017}, machine learning~\cite{andrieu2003introduction} and diffusion models for generative artificial intelligence \cite{DiffModelHigham25}. A typical objective is to compute expectations in a large dimension \(d\) of the form 

\[
\pi(f) \coloneqq \mathbb E_\pi[f] = \int_{\mathbb{R}^d} f(\mathbf{x})\, \pi (\mathbf{x})d\mathbf{x}, \quad f: \mathbb{R}^d\to \mathbb{R}, \: f \in L^1(\pi) .
\]
In high dimensions, deterministic quadrature formulas become prohibitively expensive. Markov chain Monte Carlo methods instead construct a stochastic process \((\mathbf{X}_t)_{t\ge0}\) that is ergodic with respect to \(\pi\). Then, under appropriate assumptions, for any integrable function $f$, the temporal average of the process satisfies a strong law of large numbers and a central limit theorem,
\begin{equation*}
\pi_f(t) \coloneqq \frac{1}{t} \int_0^t f(\textbf{X}_s) \,ds  \xrightarrow[\;t \to \infty\;]{\text{a.s.}} \pi(f), \quad \sqrt{t} \left( \pi_f(t) - \pi(f) \right) \xrightarrow[\;t \to \infty\;]{\text{d}} \mathcal{N}(0,\sigma_f^2), \label{temporal average}
\end{equation*}
where \(\sigma_f^2\) is the asymptotic variance \cite{CLTKipnis86}. These two results highlight two complementary criteria of sampling efficiency. The convergence of the law of \(\mathbf{X}_t\) toward \(\pi\) determines the duration of the transient or burn-in regime, whereas \(\sigma_f^2\) determines the statistical efficiency of long-time averages. A third consideration arises after discretization: different dynamics may require substantially different computational costs per time step and induce a bias. A standard choice to define a process \((\mathbf{X}_t)_{t\ge0}\) ergodic w.r.t \(\pi = Z^{-1} e^{-V}\), where \(V\) is a given smooth potential on \(\textbf{x}\in\mathbb R^d\) and \(Z\) is a normalization constant, is the overdamped Langevin dynamics, which are also reversible w.r.t \(\pi\),
\begin{equation}
    d\mathbf{X}_t = -\nabla V(\mathbf{X}_t) \, dt + \sqrt{2} \, d\mathbf{W}_t,
    \label{OLD}
\end{equation}
where \(\mathbf{W}_t\) is a standard \(d\)-dimensional Wiener process. Often the underdamped Langevin dynamics, where one also considers the momentum variable \(P_t\), is used~\cite{CKLP_2023}
\begin{equation}
    \begin{cases}
        d\mathbf{X}_t = \mathbf{P}_t\,dt,\\
        d\mathbf{P}_t = - \nabla V(\mathbf{X}_t)\,dt - \Gamma \mathbf{P}_t \, dt + \sqrt{2 \Gamma} \, d\mathbf{W}_t,
    \end{cases} \label{ULD}
\end{equation}
where \(\Gamma\) is a positive definite friction matrix. The density of the invariant distribution of the process \eqref{ULD} is proportional to \( \exp\left( -V(\mathbf{x})-\frac{1}{2}\lvert\mathbf{p}\rvert^2 \right)\). In particular, the marginal in the position variable corresponds to the target measure \(\pi\). Several different Langevin-based dynamics that can be used for sampling from $\pi(x)$ can be found in~\cite{Duncan2017}. 

A major issue for the design of an efficient sampler based on Langevin dynamics \eqref{OLD} or \eqref{ULD} is the rate of convergence of \(\pi_t \to \pi\) as \(t\to \infty\), which dictates the length of the time interval length considered for discretizing the dynamics.
Precisely, the rate of convergence to equilibrium of the continuous-time dynamics \eqref{OLD}, when measured in the $\chi^2$ divergence and various other norms, is governed by the Poincar\'{e} constant of the generator of the Langevin dynamics. Indeed, under technical assumptions, one can show that the density \(\pi_t\) of the distribution of the process \(X_t\) satisfies 
\begin{equation}
    \lVert \pi_t - \pi \rVert_{L^2(\pi^{-1})} \leq e^{-\lambda t } \lVert \pi_0 - \pi  \rVert_{L^2(\pi^{-1})}, \label{Exponential decay}
\end{equation}
where $\lambda$ is the spectral gap in $L^2(\pi)$ of the generator of the dynamics \eqref{OLD}. We assume in this paper that the initial distribution \(\pi_0\) is in the weighted Hilbert space \( L^2(\pi^{-1})\) or one can consider for simplicity a deterministic initial condition \(\textbf{X}_0 = \textbf{x} \in \mathbb R^d\). In favorable cases, \(\lambda\) is a constant of size one and a relatively short time interval can be considered. However, in many situations, the size of \(\lambda\) can be very small, which then requires a dramatically long time interval length \(T =O( 1/\lambda) \gg 1\). This issue arises in particular in multimodal distributions, where transitions between modes can be rare. Furthermore, it also appears in unimodal distributions when the potential exhibits a multiscale structure. 

Moreover, for such multiscale potential, an additional difficulty arises for the design of numerical sampler, where stiffness issues lead to a severe time step restriction for standard methods such as the simplest unadjusted Langevin algorithm (ULA). It corresponds to the Euler-Maruyama method applied to the overdamped Langevin dynamics \eqref{OLD},
\begin{equation}
    \mathbf{X}_{n+1} = \mathbf{X}_n-\nabla V(\mathbf{X}_n) h + \sqrt{2 h} \, \mathbf{\xi}_n, \quad \mathbf{\xi}_n \overset{i.i.d}{\sim} \mathcal{N}(0,I).
    \label{ULA}
\end{equation}
Hence, the design of more efficient numerical samplers is a very active research area and many advanced samplers have been derived not only using more efficient numerical integrators but also modifying the continuous dynamics itself.

A well studied modification of the overdamped Langevin dynamics \eqref{OLD}, proposed in \cite{AccGausHwang93} and further analyzed in the Gaussian case in \cite{NRPLelievre2013}, is to add a nonreversible drift perturbation \(g: \mathbb{R}^d \to \mathbb{R}^d \) to the equation \eqref{OLD},
\begin{equation}
d\mathbf{X}_t =  -\nabla V \left( \mathbf{X}_t \right) \,dt + g\left( \mathbf{X}_t \right) \, dt + \sqrt{2} \, d\mathbf W_t. \label{nonreversible SDE} 
\end{equation}
In order to preserve the invariant distribution of the original dynamics \eqref{OLD}, \(g\) is chosen divergence--free with respect to the invariant distribution. That is
\begin{equation}
    \nabla \cdot(g\, \pi) = 0. \label{g divergence free}
\end{equation}
A natural choice of perturbation $g$ satisfying the above divergence-free condition is \( g(\mathbf{x}) = J \nabla V(\mathbf{x}) \), where \(J\in \mathbb R^{d\times d}\) is a skew symmetric matrix. In this case, the perturbation $g(\textbf{X}_t)\,dt$ corresponds to a Hamiltonian-like dynamics whose flow preserves both the potential \(V\), since \(g\cdot\nabla V=0 \), and the volume, since \(\nabla \cdot g = 0\).

Under a careful choice of \(J\), the nonreversible dynamics \eqref{nonreversible SDE} converges faster to equilibrium with new efficient samplers and analysis proposed in \cite{Duncan2017}, \cite{OptGaussWu14} and \cite{zhu2026optimizingirreversibleperturbationsunadjusted}. Moreover, one can also optimize \(J\) to reduce the asymptotic variance when interested in a specific observable \(f\) as presented in  \cite{VarreducNRP16}. Building on these theoretical advantages, these nonreversible dynamics have recently found applications in generative diffusion models as presented in \cite{lu2026steeringdynamicalregimesdiffusion}. Note that other choices of drift perturbation have been studied, see for instance \cite[Section~2]{NRPSplittingDuncan17} and \cite{enhancFlow}. However, such nonreversible perturbations also have mathematical and numerical drawbacks. In particular, the exponential decay \eqref{Exponential decay} becomes
\begin{equation}
    \lVert \pi_t^g - \pi \rVert_{L^2(\pi^{-1})} \leq C(V,g)\: e^{- \tilde \lambda t } \lVert \pi_0 - \pi  \rVert_{L^2(\pi^{-1})}, \label{Exponential decay NRP}
\end{equation}
where \(\pi_t^g\) denotes the distribution of the process \eqref{nonreversible SDE} and \(C(V,g)\) is a constant in time that depends on the potential and the perturbation and can possibly be very large. While one improves the spectral gap \(\lambda\) with a larger value \(\tilde \lambda\), the constant \(C\) in \eqref{Exponential decay NRP} worsens the transient phase in the convergence to equilibrium. Moreover, improving the spectral gap does not necessarily simultaneously optimize the asymptotic variance reduction and the convergence rate to equilibrium as studied in~\cite{VarreducNRP16}.

To preserve the reversibility of the perturbed dynamics, another widely studied approach is the preconditioned Langevin dynamics. Let \(M(\mathbf{x}) = D(\mathbf{x})D(\mathbf{x})^T\) be a matrix valued function for a diffusion tensor \(D(x)\) satisfying technical assumptions. The preconditioned Langevin dynamics is given by~\cite{OptidiffuLangevinLe26, Pavl2014, ayala2025reversibility}
\begin{equation}
    d\mathbf{X}_t = - M(\mathbf{X}_t) \nabla V(\mathbf{X}_t) \, dt + \nabla \cdot M(\mathbf{X}_t) \,dt + \sqrt 2D(\mathbf{X}_t)d\mathbf{W}_t. \label{Prec}
\end{equation}
The divergence correction term \(\nabla \cdot M = (\nabla \cdot M_{i,:})_{i=1}^d\) is necessary to preserve \(\pi\) when the preconditioner depends on the position. The role of \(M\) is to modify the local geometry of the dynamics. Directions in which the target varies strongly can be rescaled differently from directions in which it varies weakly, potentially reducing anisotropy and stiffness.
The preconditioned Langevin dynamics~\eqref{Prec} have been widely studied, and the recent works \cite{OptidiffuLangevinLe26} and \cite{cui2025optimalriemannianmetricpoincare} provide an algorithm to compute the optimal preconditioner in the non-convex case. However, while the convergence rate can be significantly improved, computing the optimal preconditioner remains difficult in practice. When \(V\) is strictly convex, it has been known for a long time, and it follows directly from the Brascamp-Lieb inequality~\cite{Helffer1998}, that the  optimal preconditioner studied, up to a normalization constant, is the inverse of the Hessian, $M(x)~=\left(\nabla^2V(x)\right)^{-1}$; see \cite{State_depFrictionULDTaoMolei25} for a recent study.

However, in many cases, \(M\) and its divergence \(\nabla \cdot M\) are not easily accessible and are computationally expensive to approximate numerically. Furthermore, even though numerical methods were investigated to reduce the computational cost, see for instance \cite{MetropolisDiffuRabee14},  simulating \eqref{Prec} is still much more expensive than \eqref{OLD} and \eqref{nonreversible SDE}. Recent works aimed at avoiding the computation of an expensive preconditioner that requires the calculation of gradients include \cite{SamAdamBen25} and its extension in \cite{SAPALBen26}, where adaptive time steps and a varying diffusion tensor are used.

Towards the best of both worlds of preserving the reversibility of the dynamics and avoiding divergence terms in the dynamics, an intermediate construction is presented in \cite{StraoAbd19}. The idea is to use a one dimensional Stratonovich noise perturbation instead of a drift perturbation. Let \( \beta_t \) be a one dimensional Wiener process independent of the \(d\)-dimensional Wiener process \(W_t\). The dynamics with the Stratonovich perturbation are given by
\begin{equation}
    d\mathbf{X}_t =  -\nabla V \left( \mathbf{X}_t \right) dt + \sqrt{2}\,g\left( \mathbf{X}_t \right) \circ d\beta_t + \sqrt{2} \, d\mathbf W_t. \label{Stratonovich SDE}
\end{equation}
Interestingly, when \(g\) is divergence free w.r.t \(\pi\), i.e. satisfying \eqref{g divergence free}, we can think of \eqref{Stratonovich SDE} as a particular case of \eqref{Prec}, where the effective diffusion tensor \(M\) and coefficient matrix \(D\) are given by
\[
M(\mathbf{x}) = I + g(\mathbf{x}) g(\mathbf{x})^T \in \mathbb R ^{d\times d}, \quad D(\mathbf{x}) = \begin{pmatrix}
    I&g(\mathbf{x})
\end{pmatrix}\in\mathbb R ^{d\times d+1}.
\]
Notably, another link between \eqref{nonreversible SDE} and \eqref{Prec} was introduced in \cite{Geom_infNRPZhang22}. The authors built a space dependent nonreversible perturbation which takes into account the preconditioner \(M\).
%
%
\subsection*{Our contributions}

The aim of this paper is to understand how to build a Stratonovich perturbation to the Langevin dynamics of the form \eqref{Stratonovich SDE} that quantitatively improves the convergence rate. We provide a theoretical analysis in the case where the potential \( V\) is quadratic, $ V(\mathbf x) = \frac{\mathbf x^T S \mathbf x}{2}$, where \(S\) is a real symmetric positive definite matrix.
In particular, we develop an algorithm for constructing Stratonovich perturbations improving the spectral gap. Our main contribution can be summarized in the following theorem.
\begin{theorem}[Quantitative convergence rates] 
Consider the dynamics
\begin{equation}
    d\mathbf{X}_t =  -S  \mathbf{X}_t\, dt + \sqrt{2}\, J S \, \mathbf{X}_t  \circ  d\beta_t + \sqrt{2} \, d\mathbf W_t. \label{Stratonovich quadratic SDE}
\end{equation}
Given a symmetric positive definite matrix \(S \in \mathbb R^{d\times d}\), there exists a constructive algorithm for building a skew-symmetric matrix \(J_d\) so that the spectral gap of the generator associated with the dynamics \eqref{Stratonovich quadratic SDE} is bounded below by an explicit constant \(\gamma_d \). We have
    \[
    \lVert \pi_t - \pi \rVert_{L^2(\pi^{-1})} \leq e^{-\gamma_d t } \lVert \pi_0 - \pi  \rVert_{L^2(\pi^{-1})}.
    \]
 In particular, in the two dimensional case, when
    \[
    S_2 = \begin{pmatrix}
        1 & 0 \\ 0 & \lambda
    \end{pmatrix}, \: J_2 = \frac{1}{\sqrt \lambda} \begin{pmatrix}
       0 & 1\\ -1 & 0
    \end{pmatrix}, \: 0<\lambda \ll 1,
    \]
    the convergence rate of the Stratonovich dynamics \eqref{Stratonovich quadratic SDE} induced by \(J_2\) is at least \(\gamma_2 = 3 - \sqrt{5}\).
\end{theorem}

We also investigate numerically the practical advantages of the Stratonovich-perturbed sampler in comparison with the existing Langevin-based samplers in Section \ref{Numerical experiment}. Our numerical experiments suggest that the acceleration of the convergence rate also holds for non-Gaussian target distributions.
\subsection*{Notation used in the paper}

We gather here the main notation used throughout the document. Unless otherwise specified, all functions are assumed to have real values, and all vectors are column vectors in $\mathbb{R}^d$. We write vectors with bold letters (e.g. $\mathbf{x} \in \mathbb{R}^d$) and uppercase letters (e.g. $B \in \mathbb{R}^{d \times d}$) denote matrices. In particular, we denote \(I\) the identity matrix.

\textbf{Matrix set:} We denote the set of symmetric positive definite matrices in $\mathbb{R}^{d \times d}$ by $\mathcal{S}_d^{>0}(\mathbb{R})$ and the set of real antisymmetric matrices is denoted by $\mathcal{A}_d(\mathbb{R})$.

\textbf{Spectrum and smallest eigenvalue:} let $B \in \mathbb{R}^{d \times d}$ be a matrix with real eigenvalues. We denote its spectrum by \(\sigma(B)\) and its smallest eigenvalue \(\Lambda(B)\). If \(B\) is symmetric, the smallest eigenvalue is given by the \emph{Rayleigh quotient}
\[
\inf_{\mathbf{x}\in \mathbb R^d \backslash{\mathbf{\{0\}}}} \frac{\inner{x}{Bx}}{\inner{x}{x}} = \Lambda(B).
\]

\textbf{Functional space:} we denote $\mathcal{P}$ the space of polynomials from \(\mathbb R^d\) to \(\mathbb R \) and \(\mathcal{P}^k\) the space of polynomials of degree at most \(k\).
\begin{definition}[Weighted $L^2$ Space with Respect to a Probability Measure]
Let $\pi$ be a probability measure on $\mathbb{R}^d$. \(L^2( \mathbb{R}^d,\pi)\) is the Hilbert space
\[
L^2( \mathbb{R}^d,\pi) \coloneqq \left\{ f : \mathbb{R}^d \rightarrow \mathbb{R} \: \bigg \lvert \int_{\mathbb{R}^d} |f(x)|^2 \, d\pi(x) < \infty \right\} \backslash \sim,
\]
where \(f\sim g \iff f = g \: a.s.\).
To simplify, we often write \( L^2(\pi)\). Let \(\mathcal{L} \) be an operator, we denote its adjoint in \(L^2(\mathbb{R}^d,d\mathbf{x})\) by \(\mathcal{L}^*\) and its adjoint in \(L^2(\mathbb{R}^d,\pi)\) by \(\mathcal{L}^{\#}\). 
\end{definition}
\begin{example}
    Consider \(\mathcal{L} = -\nabla^{\#} \nabla\), where \(\pi\) has density proportional to \(e^{-V}\) then we have
    \[
    \mathcal{L}\phi = (-\nabla V) \cdot \nabla \phi + \Delta \phi.
    \]
    On the other hand \(\tilde{\mathcal{L}} = -\nabla^{*} \nabla\) is simply the Laplacian \(\tilde{\mathcal{L}}\phi = \Delta\phi\).
\end{example}

\textbf{Probability:} We write \(\mathbb E_\mathbf{x}[f(\mathbf{X}_t)]\) to denote the expectation of \(f(\mathbf{X}_t)\) with the initial condition \(\mathbf{X}_0 = \mathbf{x}\).

\subsection*{Outline of this work}
    In Section \ref{nonreversible and reversible perturbations}, we recall existing results about the nonreversible perturbation and the Stratonovich perturbation. Then, we state the main results of this work in Section \ref{Spectral gap analysis}, which we prove in both Section \ref{Computational approach in dimension two} and Section \ref{Proof in a general dimension: Rayleigh estimation}. In Section \ref{Numerical experiment}, we demonstrate the practical effectiveness of the Stratonovich sampler \eqref{Stratonovich SDE} via numerical experiments. Conclusions are presented in Section~\ref{sec:conclusions}.
    
\section{Nonreversible and reversible perturbations}\label{nonreversible and reversible perturbations}

\subsection{Nonreversible perturbation to improve the spectral gap} \label{NR perturbations}

In this section, we summarize the main results and ideas from \cite{NRPLelievre2013}. Consider the perturbed dynamics,
\begin{equation*}
d\mathbf{X}_t = \left( -\nabla V \left( \mathbf{X}_t \right) + g\left( \mathbf{X}_t \right) \right) dt + \sqrt{2} \, d\mathbf W_t, \qquad \text{\eqref{nonreversible SDE}} 
\end{equation*}
where \(g\) is chosen divergence free with respect to the invariant distribution in order to preserve the invariant distribution. That is
\[
\nabla \cdot(g e^{-V}) = 0.
\]
A general method to construct a  divergence free perturbation is to define \( g = J \nabla V \), where \( J \) is a constant antisymmetric matrix. We can note that the generator of these dynamics can be written as 
\[
\mathcal{L}_g = -\nabla^{\#}\nabla + g \cdot \nabla = \mathcal{L} + \mathcal{A}_g.
\]
The original overdamped Langevin dynamics (corresponding to \(g=0\)) are reversible in the sense of Markov processes. This is reflected in the fact that their generator \(\mathcal L\) is self-adjoint in \(L^2(\pi)\). On the other hand, the perturbed generator \(\mathcal L_g\) of the dynamics \eqref{nonreversible SDE} is not self-adjoint in \(L^2(\pi)\). In fact, we have
\[
\mathcal L_g   ^{\#} =\mathcal L_{-g}.
\]
As a result of the non reversibility, the spectrum of the operator may include complex eigenvalues, which induce an oscillatory behavior.

Consider the quadratic case
\begin{equation}
    dX_t = -S\,\textbf{X}_t\,dt + JS\,\textbf{X}_t\,dt + \sqrt{2}\,d\textbf{W}_t, \label{NRP quadratic}
\end{equation}
where \( S\in \mathcal S_d^{>0}(\mathbb R) \) and \(J \in \mathcal{A}_d(\mathbb{R})\). In \cite{Metafune2002}, it is shown that the spectrum of the generator of the dynamics \eqref{NRP quadratic} in $L^p(\pi)$ for $p\in (1,\infty)$ is
\[
\sigma(\mathcal{L}_g) = \left\{ -\sum_{j=1}^{r} n_j \lambda_j \;\middle|\; n_j \in \mathbb{N}, \lambda_j \in \sigma(S) \right\},
\]
where $\{\lambda_j\}_{j=1}^{r}$ are the $r$ distinct eigenvalues of $S+JS$. See also~\cite{ OttobrePavliotisPravda12, OPP2015}. Note that the result in \cite{Metafune2002} implies that the spectral gap of the generator $\mathcal{L}$ is the smallest eigenvalue of \(S\). 
Hence, introducing an appropriate \(J\) allows to drastically improve the spectral gap, as stated by the following result.

\begin{theorem*} [\cite{NRPLelievre2013}] \label{Theorem NR}
    Define $B_J = (I + J)S$. Then
\begin{equation}
    \max_{J \in A_d(\mathbb{R})} \min \mathrm{Re} \ \sigma(B_J) = \frac{\mathrm{Tr}(S)}{d}. \label{19}
\end{equation}
 Furthermore, there is a constructive algorithm to construct matrices $J_{\text{opt}} \in A_d(\mathbb{R})$ such that the maximum in \eqref{19} is attained. The matrix $J_{\text{opt}}$ can be chosen so that the semigroup associated with $B_{J_{\text{opt}}}$ satisfies the bound
\[
\lVert e^{-(I + J_{\text{opt}})St} \rVert \leq \frac{C_d^{(1)}}{d} \kappa(S)^{1/2} \exp \left( - \frac{\mathrm{Tr}(S)}{d} t \right),
\]
for all $t \geq 0$, where the matrix norm is induced by the Euclidean norm on $\mathbb{R}^d$ and $\kappa(S) = \|S\| \|S^{-1}\|$ denotes the condition number.
\end{theorem*}

Next, we present a result which is essential to understand how the algorithm builds an optimal perturbation. We cite two equivalent conditions, but more are presented in \cite[Section~3]{NRPLelievre2013}.

\begin{proposition*}[\cite{NRPLelievre2013}]
    Assume that $J \in \mathcal{A}_d(\mathbb{R})$, define \(\tilde J = S^{1/2} JS^{1/2} \) and assume that $S \in \mathcal{S}^{>0}_d(\mathbb{R})$. Given a matrix $Q \in S^{>0}_d(\mathbb{R})$, we denote $\{\lambda_k\}_{k=1}^d$ the positive real eigenvalues of $Q$ (counted with multiplicity), and $\{\psi_k\}_{k=1}^d$ the associated eigenvectors, which form an orthonormal basis of $\mathbb{R}^d$.
    
Then the following conditions are equivalent
\begin{itemize}
    \item[(i)] The matrix $\tilde{B}_J = S + \tilde{J}$ is diagonalizable in $\mathbb{C}$ and its spectrum satisfies $\sigma(\tilde{B}_J) \subset \frac{\text{Tr}(S)}{d} + i\mathbb{R}$.
    \item[(ii)] There exists \(Q\in S^{>0}_d(\mathbb{R})\), such that for all $k \in \{1, \dots, d\}$,
\begin{equation}
(\psi_k, S \psi_k)_{\mathbb{R}} = \frac{\operatorname{Tr}(S)}{d}, \label{36}
\end{equation} 
and, for all $ j \neq k \in \{1, \dots, d\}$,
\begin{equation}
(\lambda_j - \lambda_k)(\psi_j, \tilde{J} \psi_k)_{\mathbb{R}} = (\lambda_k + \lambda_j)(\psi_j, S \psi_k)_{\mathbb{R}}. \label{37}
\end{equation}
    
\end{itemize}
\end{proposition*}

When the eigenvalues of $Q$ all have multiplicity one, $\tilde{J}$ is completely determined by \eqref{37}
\begin{equation}
\psi_j^T \tilde{J} \psi_k = -\frac{\lambda_k + \lambda_j}{\lambda_k - \lambda_j} \psi_j^T S \psi_k , \text{ for all } j\neq k \in \{1,\dots,d \}.  \label{38}
\end{equation} 
Indeed, by the antisymmetry of $\tilde{J}$, the remaining entries are zero
\begin{equation}
\psi_j^T \tilde{J} \psi_j= 0 \quad \text{for all } j \in \{1, \dots, d\}.
\end{equation}
The only issue is to build an appropriate basis, but this task is fulfilled by the algorithm presented in \cite[Section~3]{NRPLelievre2013}. The idea is to equilibrate the action of \(S\) on an arbitrary basis by iteratively pairing small and large modes and using Gram-Schmidt to keep the basis orthonormal.

To summarize, we can construct a basis $\{\psi_k\}_{k=1}^d$ such that \eqref{36} is fulfilled. Then, we can arbitrarily choose positive and distinct eigenvalues $\{\lambda_k\}_{k=1}^d$. Finally, we can construct a matrix $\tilde{J}$ so that $S^{-1/2}\tilde{J}S^{-1/2}$ is an optimal perturbation in the sense that it maximizes the spectral gap.

Note that it is not entirely clear how the choice of the eigenvalues \(\{\lambda_k\}_{k=1}^d\) affects the imaginary part of the spectrum of \(\mathcal L_g^*\). Moreover, this construction involves using repetitively Gram-Schmidt and is costly.

\subsection{Reversible perturbation to improve the spectral gap} \label{R Perturbation}

While adding a nonreversible perturbation can improve the convergence to equilibrium, it also induces an oscillatory and transient behavior. In \cite{StraoAbd19},  it is proposed is to use a one dimensional Stratonovich noise instead of perturbing the drift directly. Consider the dynamics \eqref{Stratonovich SDE}
\begin{equation*}
    d\mathbf{X}_t =  -\nabla V \left( \mathbf{X}_t \right) dt + \sqrt{2}\, g\left( \mathbf{X}_t \right) \circ d\beta_t + \sqrt{2} \, d\mathbf W_t. \qquad \text{\eqref{Stratonovich SDE}}
\end{equation*}
\begin{theorem*}[\cite{StraoAbd19}]
    The generator of \eqref{Stratonovich SDE} is
    \[
    \mathcal{R}= \mathcal{L} + \mathcal{A}_g^2,
    \]
    where \(\mathcal{L} = -\nabla V \cdot \nabla + \Delta\) is the generator of \eqref{OLD} and \(\mathcal {A}_g  = g \cdot \nabla\) is the generator of the perturbation. If \(g\) is divergence free under \(\pi\), i.e \(\nabla \cdot (g \,\pi) =0\), then \(\mathcal{R}\) is self-adjoint in \(L^2(\pi)\) which is equivalent to the fact that the dynamics \eqref{Stratonovich SDE} is reversible w.r.t \(\pi\). Moreover,
    \[
    \lambda_\mathcal{R} \ge \lambda_\mathcal{L}, \text{ and }  \sigma_\mathcal{R} \le \sigma_\mathcal{L},
    \]
    where \(\lambda_{(\cdot)}\) is the spectral gap and \(\sigma_{(\cdot)}\) is the asymptotic variance of the corresponding generator.
\end{theorem*}

When \(g = J\nabla V\), the operator \( \mathcal{R} \) is self-adjoint in \( L^2(\pi) \), which implies that its spectrum is real. Consequently, the convergence to equilibrium remains non-oscillatory, in contrast to the nonreversible perturbation case where complex eigenvalues can induce oscillations (see Section \ref{Numerical experiment} for illustrations of this phenomenon).

A crucial observation that was made in~\cite{StraoAbd19} is that any perturbation of this form does not worsen the convergence. However, this article only gives qualitative improvements and does not quantify them. In particular, it is not investigated whether there is an optimal convergence rate and if we can build an optimal matrix \(J\).

\section{Spectral gap analysis}\label{Spectral gap analysis}
In this section, we present the main result concerning the analysis of the dynamics
\begin{equation*}
    d\mathbf{X}_t =  -S  \mathbf{X}_t\, dt + \sqrt{2}\, J S \, \mathbf{X}_t  \circ  d\beta_t + \sqrt{2} \, d\mathbf W_t, \qquad \text{\eqref{Stratonovich quadratic SDE}}
\end{equation*}
where \(S\) is a real symmetric positive definite matrix.
We start by exhibiting an \emph{optimal} scaling in two dimensions in Section \ref{2D Results}. The scaling is optimal in the following sense:

In two dimensions, the only liberty to build a perturbation defined by a skew-symmetric matrix \(J\) is via a scaling parameter. Indeed, all two dimensional skew-symmetric matrices are proportional to the symplectic matrix
    \[
    J_2 = \begin{pmatrix}
        0& 1 \\ -1 & 0
    \end{pmatrix}.
    \]
    From Section \ref{Computational approach in dimension two}, it appears that the best spectral gap is achieved when considering \(J = \delta \, J_2\) and taking \(\delta \to \infty\). This is not surprising in comparison with the analysis in \cite{NRPLelievre2013}, where they obtain a similar conclusion.
    
    For obvious reasons, taking \(\delta \to \infty\) is not feasible numerically. Moreover, the dynamics induced by the perturbation \(\dot z = \delta J_2\, \nabla V(z)\) is a Hamiltonian system, whose highly oscillatory behavior is difficult to handle in practice. Since the oscillatory behavior is strengthened when \(\delta\) is large, it is desirable to minimize the size of \(\delta\).
    
    For these reasons, we define the optimal scaling as the smallest \(\alpha >0\) such that the convergence rate of \eqref{Stratonovich quadratic SDE} with the skew-symmetric matrix  \(\delta J_2\), where \(\delta = \lambda^{-\alpha}\), remains bounded away from zero as \(\lambda \to0^+\).

    In Section \ref{Pairwise matching in a general dimension}, we consider a specific matrix \(J_d\) in a general dimension \(d\). By considering the scaling derived in two dimensions, we show that, asymptotically, the convergence rate remains bounded away from \(0\) under some hypothesis on the spectrum of \(S\). However, the convergence remains slow in some pathological examples. In order to address this issue, we propose another construction in Section \ref{Splitting eigenvalues in a general dimension}.

\subsection{Optimal scaling in dimension two} \label{2D Results}
We start by presenting the set up in two dimensions, then we present the main result of this section.

We consider the quadratic stiff potential
\[
V_\lambda(\mathbf{x}) = \frac{1}{2} (x_1^2 + \lambda x_2^2) = \frac{1}{2}
\textbf{x}^T \begin{pmatrix} 1 & 0 \\ 0 & \lambda \end{pmatrix} \textbf{x} =\frac{1}{2}  \textbf{x}^T S \textbf{x}, \quad 0<\lambda \ll 1.
\]
We introduce a Stratonovich perturbation of the form 
\[
g(\textbf{x}) = \delta  J \nabla V(\textbf{x}) = \delta  \begin{pmatrix} 0 & 1 \\ -1 & 0 \end{pmatrix} 
\begin{pmatrix} x_1 \\ \lambda x_2 \end{pmatrix} 
= \begin{pmatrix}  \lambda \delta x_2 \\ - \delta x_1 \end{pmatrix}.
\]
The resulting process is the solution of the following Itô SDE
\begin{equation}
    d\mathbf{X}_t  = -B  \mathbf{X}_t  dt +\delta J S \mathbf{X}_t \sqrt{2} d\beta_t + \sqrt{2} \, d\mathbf{W_t}, \label{Reversible SDE 2d}
\end{equation}
where $B = B_{\delta} = S + \delta^2 \lambda I = \begin{pmatrix}
    1 +  \delta^2 \lambda & 0 \\
    0 & (1 +  \delta^2) \lambda
\end{pmatrix}$.

    \begin{theorem} \label{Main Result 2d}
Let \(\lambda > 0\), and consider the perturbed stochastic differential equation~\eqref{Reversible SDE 2d} with parameters \(\delta = \lambda^{-1/2}\). Then, the spectral gap of the perturbed operator \(\mathcal{R}\) is at least \(\gamma = 3 - \sqrt{5}\). In particular, for all \(t \geq 0\), we have
\begin{equation}
        \lVert e^{t\mathcal{R}}\phi - \pi(\phi) \rVert_{L^2(\pi)} \leq  e^{-\gamma t} \: \lVert \phi - \pi(\phi) \rVert_{L^2(\pi)} , \: \text{ for all } \phi \in L^2(\pi). 
\end{equation}
\end{theorem}

\subsection{Pairwise matching in a general dimension}\label{Pairwise matching in a general dimension}
We consider the case where the dimension is even \(d=2n \in 2\mathbb{Z}_{>0}\). Let us suppose that \[S=\mathrm{diag}(s_1,\dots,s_d), \quad 1 \geq s_1 \geq \dots \geq s_d>0,\] 
and let us consider 
\begin{equation}
    J = J(S) = \begin{pmatrix}
    &&&&&& a_1  \\ &&&& \dots  \\ &&& a_{n} \\   && -a_{n} \\ & \dots \\ -a_1
\end{pmatrix}, a_i = \frac{1}{\sqrt{s_{i}s_{d-i+1}}}, \: i=1,\dots,n. \label{J pairwise matching}
\end{equation}

Based on the results in Section \ref{2D Results}, we expect to have a convergence rate bounded below proportionally to
\[
\gamma_{\min} = \min_{\substack{i = 1,\dots,n}}  s_i +  s_{d+1-i}.
\]
And we expect that we can only reach this bound asymptotically. We confirm in the following results that this phenomena occurs in a general dimension.

\begin{theorem} \label{Main result d=2n asymptotic}
Let us denote \(J_{\hat \delta} = {\hat \delta} J\), where \(J\) is of the form \eqref{J pairwise matching}. Let \(\lambda_{\hat \delta}\) be the spectral gap of the generator associated with the drift matrix \(S\) and the perturbation defined by \(J_{\hat \delta}\). We have
\[
\lim_{{\hat \delta} \to \infty} \lambda_{\hat \delta}  = \gamma_{\min}.
\]
\end{theorem}

\begin{theorem} \label{Main result d=2n non asymptotic}
    In the notation of Theorem \ref{Main result d=2n asymptotic},
  If 
        \[
    \hat \delta^2 \ge \max \left(\frac{ (\gamma_{\min})^2 + 16s_1^2 - 4 \gamma_{\min} s_d }{4\gamma_{\min}s_d}, 1\right),
    \]
    Then 
    \[
    \lambda_{\hat \delta} \geq \frac{\gamma_{\min}}{2}.
    \]
\end{theorem}

\begin{remark}
 One should compare the convergence rate \(\gamma_{\min}\) with the optimal convergence rate in the nonreversible case, which is \(\text{Tr}(S)/d\). When the eigenvalues \(s_1 \ge \dots\ge s_d\) are approximately distributed symmetrically around their mean, i.e \( \gamma_{\min}/2 \approx \text{Tr}(S)/d\), then the two convergence rate are of the same order. For instance, when \(S\) is the discrete Laplacian \eqref{discrete Laplacian}, we have exactly for all \(i = 1,\dots,n\)
    \[
    \frac{s_i + s_{d+1-i}}2 = \frac {\gamma_{\min}} 2= \frac{\text{Tr}(S)}d.
    \]
    On the other hand, if \(S\) has a few big eigenvalues and a lot of small ones, the Stratonovich perturbation induced by the matrix \(J\) from \eqref{J pairwise matching} does not significantly improve the convergence. For instance in the worst case \(S = \text{diag}(1,\lambda,\dots,\lambda)\), we observe that the spectral gap becomes $\lambda \mapsto 2\lambda$ so the perturbation does not improve the situation. Such experiment is presented in Section \ref{Multidimensional Gaussian experiment}.
    
     Another kind of problem arises when considering a matrix $S$ with a lot of small eigenvalues: the coefficients $a_i = \frac{1}{\sqrt{s_i \, s_{d+1-i}}}$ may be large causing issues for numerical implementation notably of stiffness and step-size restrictions.

      We address these issues in Section \ref{Splitting eigenvalues in a general dimension} by proposing another construction of perturbation.
\end{remark}

\begin{remark}
The statement of Theorem \ref{Main result d=2n non asymptotic} is valid for a rather large perturbation scaling \(J_{\hat \delta} = \hat \delta J = O(s_d^{-1})\). This contrasts with the two dimensional optimal scaling \(O(s_d^{-1/2})\) and we observe in numerical experiments of Section \ref{Numerical experiment} that choosing \(\hat \delta = 1\) offers significant improvement in terms of convergence rate. The reason of this theoretical restriction is further explained in Section \ref{Proof in a general dimension: Rayleigh estimation}. 
\end{remark}

 We now describe the underlying algorithm to build \(J\) in a more general setting, which allow to build a reversible sampler given a potential \(V\) by considering the dynamics
 \begin{equation}
     d \textbf{X}_t = -\nabla V(\textbf{X}_t)\,dt + J \nabla V(\textbf{X}_t) \circ \sqrt{2}d\beta_t + \sqrt{2}d \textbf{W}_t. \label{General Strato Sampler}
 \end{equation}

\begin{algorithm} [H]
\caption{Constructing a pairwise-matching perturbation}
\label{Aglo pairwise}
\begin{algorithmic}[1]

\State \textbf{Input:} $A\in\mathbb{R}^{d\times d}$ symmetric positive definite.

\State Diagonalize \(A\) as \(Q A Q^T = D\), where \(D=\textrm{diag}(s_1,\ldots,s_d)\) is a diagonal matrix containing eigenvalues \(s_1 \ge \dots \ge s_d\) sorted in non increasing order and \(Q\) is an orthogonal matrix.

\State Define \(\tilde J= J( D)\) using \eqref{J pairwise matching}.
 \State \Return \(J = Q^T  \tilde J \, Q \).
\end{algorithmic}
\end{algorithm}
In the following, we denote \(J(A)\) the skew-symmetric matrix built using Algorithm \ref{Aglo pairwise} with input~\(A\).

Clearly, Algorithm \ref{Aglo pairwise} applies naturally to the case where \(V(\textbf{x}) = \frac{\textbf{x}^T A \textbf{x}}{2}\) is quadratic. In this case, we recover the theoretical guarantee of Theorem \ref{Main result d=2n asymptotic} and numerical experiments presented in Section \ref{Multidimensional Gaussian experiment} corroborate our statement. Additionally, we can build a Stratonovich perturbation when considering a potential of the form
\begin{equation}
   V(\textbf{x}) = \frac{\textbf{x}^T A \textbf{x}}{2} + \psi(\textbf{x}). \label{V semi-quadratic} 
\end{equation}
Indeed it is natural to consider the Stratonovich perturbation associated with \(J(A)\). Numerical experiments presented in Section \ref{Multidimensional non-Gaussian experiments} indicates that such methodology significantly improves the performance of the sampler derived from \eqref{General Strato Sampler}. Finally, another possibility to construct a Stratonovich perturbation using Algorithm \ref{Aglo pairwise} is to consider \(J(\mathcal{F})\), where \(\mathcal{F}\) is the \emph{Fisher information} matrix
\[
\mathcal{F} = \mathbb{E}_\pi [ \nabla V(\textbf{x}) \nabla V(\textbf{x})^T] = \mathbb{E}_\pi [\nabla^2 V(\textbf{x})].
\]
Such approach has been considered in the recent work \cite{zhu2026optimizingirreversibleperturbationsunadjusted} to build a nonreversible perturbation. A theoretical justification of applying either \(J(A)\) or \(J(\mathcal{F})\) in \eqref{Stratonovich quadratic SDE} to the non quadratic case will be explored in a future work.

\subsection{Splitting eigenvalues in a general dimension}\label{Splitting eigenvalues in a general dimension}
To avoid the problems mentioned in the previous remark, we can consider pairing small eigenvalues strictly to the biggest one. One way to do so is to add multiple Stratonovich perturbations. For instance, suppose that, for an index $k \le d/2$, we have
\[
\begin{cases}
    s_i = O(1), \text{ if } i < k. \\ s_i = o(1), \text{ if } i \geq k.
\end{cases} 
\]
We can consider a family of perturbation matrices $(J_i)_{i = k}^{d}$ of the form
\[
J_i = \begin{pmatrix}
    &&& a_i &&&  \\ &&&&&  \\ &&&&   \\ -a_i &&&  \\   &&  \\ &  \\ 
\end{pmatrix}, \text{ i.e. } (J_i)_{k,l} = \begin{cases}
    a_i,\text{ if } k = 1, l = i. \\ -a_i,\text{ if } k = i, l = 1.\\ 0,\text{ else}.
\end{cases}
\]
Where \(a_i\) are as in \eqref{J pairwise matching}. The dynamics perturbed with $d-k+1$ Stratonovich noises $(\beta_i)_{i=k}^{d} $ write as the SDE
\begin{equation}
    d\mathbf{X}_t = -\left (S - \sum_{i =k}^{d} (J_iS)^2 \right ) \mathbf{X}_t dt + \sum_{i =k}^{d}  J_iS \: \mathbf{X}_t \sqrt2 d\beta_i(t) + \sqrt 2d\mathbf W_t. \label{method2 SDE}
\end{equation}

This dynamics is reversible w.r.t the invariant measure \(\pi\). In the proof of the two previous result, we can see that the spectral gap of the generator is essentially governed by the rate of convergence to equilibrium of the second moment. Moreover, we can compute that the second moment of the dynamics \eqref{method2 SDE} converges at speed at least
\[ \lambda_{d,k} =  2 + d-k - \sqrt{(1+d-k)^2 + 2(d-k+1) - 3}.\]
We remark that \( \lambda_{d,k} \approx  \frac{2}{d-k+1}\) when \(d-k\gg 1\). Hence, an intuitive interpretation of this approach is that it \emph{splits} the strength of the largest eigenvalue into the smallest ones. Note also that \(\lambda_{2,1} = 3 - \sqrt{5}\), which corresponds exactly to Theorem \ref{Main Result 2d}.

We could make rigorous this heuristic argument by using the same proof as in Section \ref{Proof in a general dimension: Rayleigh estimation}. For the sake of brevity, we do not provide more details, but we illustrate the benefits of this approach compared to the construction \eqref{J pairwise matching} in a numerical example in Section \ref{Multidimensional Gaussian experiment}.

\begin{remark}
We could also try the same idea using only one noise by considering the matrix
\[
J = \sum_{i=k}^{d-1} J_i.
\]
    The issue is that the drift matrix \(S- (JS)^2\) has a complicated shape and is difficult to analyze, whereas the matrices in \eqref{method2 SDE} have a sufficiently nice structure to be analyzed. Moreover, numerical experiments suggests that using only one noise with the matrix \(J\) is not sufficient to obtain a spectral gap bounded away from \(0\).
\end{remark}

\section{Proof in dimension two: moment estimates} \label{Computational approach in dimension two}
The main difficulty in this setting is that, even when the initial condition is Gaussian, the new dynamics \eqref{Reversible SDE 2d} is not Gaussian contrary to the original dynamics (i.e. with \(J=0\)). Therefore, we cannot rely on controlling only the first two moments to prove the result.

In order to do so, we first establish a bound on the convergence rate of the moments. We start with the expectation in Section \ref{Expectation}, then the covariance in Section \ref{Covariance} and finally a general moment in Section \ref{Higher moments}. Finally, we extend the bound to the \(L^2(\pi)\) norm for a general observable in Section \ref{Extension to general observables}.

\subsection{Convergence of the expectation} \label{Expectation}

Let \(\mathbf{x}_t = \mathbb{E}[\mathbf{X}_t]\). By linearity of expectation and properties of the Itô integral, \(\mathbf{x}_t\) satisfies the deterministic ODE
\[
d\mathbf{x}_t = -B_{\delta}\,\mathbf{x}_t\,dt.
\]
Its solution is explicitly given by
\[
\mathbf{x}_t = e^{-B_{\delta} t}\,\mathbf{x}_0.
\]
Since \(B_{\delta}\) is diagonal, the dynamics decouple along coordinates. Writing
\[
(\mathbf{x}_t)_1 = x_1(t) \equiv x_t, 
\qquad 
(\mathbf{x}_t)_2 = x_2(t) \equiv y_t,
\]
we obtain exponential convergence to zero in each coordinate. More precisely,
\[
x_t \text{ converges at rate } 1 + \delta^2\lambda,
\qquad
y_t \text{ converges at rate } \lambda + \delta^2\lambda.
\]

Thus, the convergence speed can be made arbitrarily large by increasing \(\delta\). For example, choosing \(\delta = \lambda^{-1/2}\) yields that the smallest eigenvalue of \(B_{\delta}\) is \(1 + \lambda\).

\subsection{Convergence of the covariance} \label{Covariance}
We consider the covariance matrix by $A_t = \mathbb E [\mathbf{X}_t\mathbf{X}_t^T]$. In order to derive the ODE associated to its evolution in time, we recall that by the Itô formula
\[
d(\mathbf{X}_t\mathbf{X}_t^T) =(d\mathbf{X}_t)\mathbf{X}_t^T + \mathbf{X}_t (d\mathbf{X}_t)^T + (d\mathbf{X}_t)(d\mathbf{X}_t)^T.
\]
We obtain the equation
\begin{equation}
    \frac{d}{dt} A_t = -B A_t - A_t B + 2\delta^2 JSA_tSJ^T + 2 I,\: A(0) = A_0. \label{Covariance ODE V}
\end{equation}
A direct calculation shows that \(A_{\infty} = S^{-1}\) is a stationary solution of this equation.
 To find a solution dependent on time, let $\tilde{A}_t$ be a solution of (\ref{Covariance ODE V}) and consider $A_t\coloneqq \tilde{A}_t - A_{\infty} $. 
Clearly, this matrix satisfies the initial condition $ A_0 = \tilde{A}_0 - S^{-1}$ and the equation
\[
    \frac{d}{dt} A_t  = -B A_t - A_t B + 2\delta^2 JSA_tSJ^T.
\]
Denote $A_t = \begin{pmatrix}
            a & b \\ b & c
        \end{pmatrix} $. We obtain
        \[
      \frac{d}{dt} A_t  = \begin{pmatrix}
            2 (\lambda \delta)^2c - 2(1+ \lambda \delta^2)a & -(1+ 4\lambda \delta^2 + \lambda)b \\ -(1 + 4\lambda \delta^2 + \lambda)b
            &  2 \delta^2 a - 2 \lambda(1+ \delta^2)c
        \end{pmatrix}.
        \]
This can be rewritten as a system of ODE
\[
\begin{cases}
\frac{d}{dt} \begin{pmatrix}
    a(t) \\ c(t)
\end{pmatrix} = -2M \begin{pmatrix}
    a(t) \\ c(t)
\end{pmatrix}, \text{ where } M =  \begin{pmatrix}
            1 + \lambda \delta^2 & -\lambda^2\delta^2 \\ -\delta^2 & \lambda(1 +  \delta^2)
        \end{pmatrix},  \\ \\
\frac{d}{dt} b(t) = -(1 + 4\lambda \delta^2+ \lambda)b(t).
\end{cases}
\]
with solutions
\[
\begin{pmatrix}
    a(t) \\ c(t)
\end{pmatrix} = e^{-2Mt} \begin{pmatrix}
    a(0) \\ c(0)
\end{pmatrix}, \quad  b(t) = e^{ -(1 + 4\lambda \delta^2 + \lambda)t}b(0).
 \]
We analyze the convergence rate of $a(t)$, $b(t)$, and $c(t)$ to $0$. As $b(t)$ behaves similarly to $\mathbf{x}_t$, we focus on the term $e^{-2Mt}$ in the remainder of this section. The following lemma characterizes the critical scaling of $\delta$ required to ensure a convergence rate bounded away from zero.

\begin{lemma} \label{Smallest eigenvalue M}
    \(\Lambda(2M)\) is an increasing function of \(\lambda > 0\) and an increasing and bounded function of \(\delta\). Moreover, if we set \(\delta  = \lambda^{-\alpha} \) for \(\alpha \in \mathbb{R}^*_+\), we have
    \[
\Lambda(2M)  \xrightarrow[\lambda \to 0]{} \begin{cases}
     1, \text{ if } \alpha > 1/2, \\
     3 - \sqrt{5} \approx 0.764 , \text{ if } \alpha = 1/2, \\ 
     0, \text{ if } 0 < \alpha <  1/2.
 \end{cases}
\]
\end{lemma}

\begin{proof}[Proof of Lemma \ref{Smallest eigenvalue M}]

A direct computation shows that the eigenvalues of $M$ are
\[
\tilde{\mu}_{\pm}
= \frac{\lambda + 2\lambda\delta^2 + 1 \pm \sqrt{(2\lambda\delta^2)^2 + (\lambda - 1)^2}}{2}.
\]
We are interested in the eigenvalues of $2M$, that we denote \( \mu_{\pm} = 2\tilde{\mu}_{\pm}\). 
For all $\delta > 0$, these satisfy $\mu_+ > \mu_- > 0$. Hence, the smallest eigenvalue of $2M$ is $\mu_- = \mu_-(\delta)$. Moreover, $\mu_-(\delta)$ is increasing in $\delta$ and satisfies
\[
\mu_-(\delta) \xrightarrow[\delta \to \infty]{} 1 + \lambda.
\]

Let $\delta = \lambda^{-\alpha}$, a direct computation yields,
\[
\mu_-(\lambda,\lambda^{-\alpha})
\xrightarrow[\lambda \to 0]{}
\begin{cases}
1, & \text{if } \alpha > 1/2, \\[4pt]
2\left(\tfrac{1}{2} + 1 - \tfrac{\sqrt{5}}{2}\right), & \text{if } \alpha = 1/2, \\[6pt]
0, & \text{if } 0 < \alpha < 1/2.
\end{cases}
\]
\end{proof}
In particular, the choice $\delta = \lambda^{-1/2}$ ensures that $\mu_-$ remains bounded away from zero as $\lambda \to 0$.

\subsection{Convergence of higher moments} \label{Higher moments}

In this section, we analyze the moments of higher order $\Phi_{k,l}\coloneqq \mathbb E [\phi_{k,l}(\textbf{X}_t)]$, where $\phi_{k,l}(\textbf{x}) \coloneqq x^k y^l$.

\begin{proposition} \label{prop ODEs Moments}
    The moments \(\Phi_{k,l}\) of order \( k + l = n\) satisfy a system of ODEs of the form
 \begin{equation}
    \frac{d}{dt}\mathbf{A}_{n} = -M_n\mathbf{A}_{n} + \mathbf{B}_n, \quad \mathbf{A}_n,\mathbf{B}_n  \in \mathbb{R}^{r_n}, \: M_n \in \mathbb{R}^{r_n\times r_n}, r_n\in \mathbb N.\label{ODEs Moments}
 \end{equation}
 In particular, the moments of even order of the form \(\Phi_{2k,2n-2k}\) satisfy \eqref{ODEs Moments} with \(r_{2n} = n+1\) and
\begin{align*}
    &\mathbf{A}_{2n} = \begin{pmatrix}
    \Phi_{2n,0} \\  \vdots \\  \Phi_{2n-2k,2k} \\ \vdots \\ \Phi_{0,2n}, \end{pmatrix}, \: M_{2n} =  \begin{pmatrix}
    a_0 &  b_0 \\
     c_1 & a_1 & b_1 \\
      & \ddots & \ddots & \ddots &  \\
     &  &c_n & a_n
\end{pmatrix} ,
\end{align*}
where the entries of $M_{2n}$ are given by
\begin{align*}
     a_k &= 2n(1+ \lambda \delta^2)  + 2k(\lambda - 1 + 4(n-k)\lambda\delta^2 ),  \\
     b_k &= -(2n-2k)(2n-2k-1)\lambda^2 \delta^2 , \\
      c_k &= -(2k)(2k-1) \delta^2,
\end{align*}
for \(k=0, \dots, n\) and the entries of \(\mathbf{B}_{2n}\) for are given by 
\[
(\mathbf{B}_{2n})_k = -\frac{b_k}{(\lambda \delta)^2} \Phi_{2n-2k-2,2k} -\frac{c_k}{\delta^2} \Phi_{2n-2k,2k-2}, \quad k=0, \dots, n.
\]

\end{proposition}

\begin{proof}[Proof of Proposition \ref{prop ODEs Moments}]
    Using the Itô formula on $d\phi_{k,l}(\mathbf{X}_t)$, we find
\[
d\phi_{k,l}(\mathbf{X}_t)= \partial_x \phi_{k,l} dx_t + \partial_y \phi_{k,l} dy_t + \frac{1}{2} ( \partial_{xx} \phi_{k,l} (dx_t)^2 + 2 \partial_{xy} \phi_{k,l} (dx_t)(dy_t) + \partial_{yy} \phi_{k,l} (dy_t)^2).
\]
Recall that,
\[
\eqref{Reversible SDE 2d} \iff
\begin{cases}
    d x_t = -(1 + \delta^2 \lambda)x_tdt + \lambda \delta y_t \sqrt2 d\beta_t + \sqrt2 dW_1(t), \\
    d y_t = -(1 + \delta^2) \lambda y_tdt - \delta x_t \sqrt2 d\beta_t + \sqrt2 dW_2(t).
\end{cases}
\]
We can now explicitly compute $d\phi_{k,l}(\textbf{X}_t)$ using
\[
(dx_t)^2 = 2(\lambda^2 \delta^2 y_t^2  + 1 ) dt, \: (dx_t)(dy_t) = -2\lambda \delta^2 x_ty_t dt, \: (dy_t)^2 = 2( \delta^2 x_t^2 +1 )dt.
\]
Finally we obtain

\begin{align*}
        d\phi_{k,l}(\textbf{X}_t) =& -k (1+\delta^2 \lambda) x_t^{k}y_t^l dt + k \lambda \delta x_t^{k-1}y_t^{l+1} \sqrt2 d\beta_t + \sqrt2 k x_t^{k-1}y_t^l dW_1(t) \\
     &-l (1+\delta^2) \lambda x_t^{k}y_t^l dt - l  \delta x_t^{k+1}y_t^{l-1} \sqrt2 d\beta_t + \sqrt2 l x_t^{k}y_t^{l-1} dW_2(t) \\
     &+k(k-1) [\lambda^2 \delta^2 x_t^{k-2}y_t^{l+2} + x_t^{k-2}y_t^{l}]dt + l(l-1) [\delta^2 x_t^{k+2}y_t^{l-2} + x_t^{k}y_t^{l-2}]dt \\
     &- 2kl \lambda \delta^2 x_t^{k}y_t^{l}dt.
\end{align*}

After taking the expectation, we find out that \(\Phi_{k,l}\) satisfies the ODE
\begin{align*}
    \frac{d}{dt}\Phi_{k,l} =& -(k(1+\delta^2 \lambda) + l (1+\delta^2) \lambda + 2kl \lambda \delta^2  ) \Phi_{k,l} + k(k-1)\lambda^2 \delta^2  \Phi_{k-2,l+2}\\
        &+ l(l-1) \delta^2 \Phi_{k+2,l-2} + k(k-1) \Phi_{k-2,l} + l(l-1)\Phi_{k,l-2}.
\end{align*}
We can then write
    \[
    \frac{d}{dt}\mathbf{A}_{n} = -M_n\mathbf{A}_n + \mathbf{B}_n.
    \]
It is straightforward to verify that the moments of even order of the form \(\Phi_{2k,2n-2k}\) satisfy the ODE presented in the statement of Proposition \ref{prop ODEs Moments}.

\end{proof}

\begin{remark}\label{odd moments not treated}
    From the proof in Section \ref{Proof in a general dimension: Rayleigh estimation}, the moments containing odd powers are not in the kernel of the generator of the perturbation and from Proposition \ref{prop cv non asymptotic pairwise matching} we can conclude that their convergence rate is at least bigger than \(1 + \lambda\). Moreover, we show in the rest of this section that moments containing only even orders converge at a rate at least \(\gamma <1\). Hence, we do not need to analyze the other systems of ODEs containing odd powers.
\end{remark}

Using \eqref{ODEs Moments}, we can describe more precisely the convergence rate of the Stratonovich perturbed process. In the following proposition, we consider the generator arising from the Stratonovich perturbation
\[
\mathcal{R} = \mathcal{L} + \mathcal{A}^2.
\]
It is a reversible, self-adjoint in \( L^2(\pi)\) and negative semi-definite operator.

\begin{proposition}\label{solution of ODEs Moments}
    The spectral gap of \(\mathcal{R}\) restricted to polynomials of degree at most \(2n\) with only monomials of even degree is at least 
    \[
    \gamma = \min_{k=1,\dots,n} \Lambda(M_{2n}).
    \]
\end{proposition}
\begin{proof}
Indeed, we observe that the Stratonovich generator \(\mathcal{R}\) restricted to polynomials of degree at most \(n\) is a linear operator represented by a block diagonal triangular matrix in the monomial basis.

    Let \(n \in \mathbb{N}\) and denote \(S_n\) the set spanned by monomials of degree \(n\), i.e
    \[
    S_n = \text{span}\{ x^ky^l | k+l = n\}. 
    \]
    Let \(k,l \in \mathbb{N}\) such that \(k+l = n\). A straightforward computation similar as in Proposition \ref{prop ODEs Moments} shows that \(\mathcal{R} x^k y^l\) can be written as a sum of monomials of degree \(n\) and a sum of polynomials of degree \(n-2\). In particular, \(\mathcal{R}(S_n) \subset S_n \cup S_{n-2}\). Therefore, the linear operator \(\mathcal{R}\) restricted to polynomials of degree at most \(n\) can be represented as a block triangular matrix. In particular, its eigenvalue lies in the diagonal block, which corresponds to the matrices \(-M_n\) from Proposition \ref{prop ODEs Moments}. The spectral gap of \(\mathcal{R}\) restricted to polynomials of degree at most \(2n\) with only monomials of even degree is at least 
    \[
    \gamma = \min_{k=1,\dots,n} \Lambda(M_{2n}).
    \] 
\end{proof}

We now analyze spectral properties of \(M_{2n}\). We exploit the fact that a change of basis does not alter the spectrum. Therefore, we select appropriate basis in which the spectral properties of \(M_{2n}\) become apparent.

\begin{lemma}\label{M_n similar to T_n}
    In the notation of Proposition \ref{prop ODEs Moments}, the matrix $M_{2n}$ associated with moments containing only even order is similar to a symmetric tri-diagonal matrix $T_{2n}$. Namely, there exists a invertible matrix \(P\) such that
    \[
    P^{-1} M_{2n} P  = T_{2n}.
    \]
    In particular, $M_{2n}$ has real eigenvalues for all \(n\).
\end{lemma}
\begin{proof}[Proof of Lemma \ref{M_n similar to T_n}]
    Consider the change of basis induced by the following matrix 
    \[ P \coloneqq \text{diag}(\delta_0, \dots,\delta_n), \text{ with }
\delta_k = \begin{cases}
    1, \text { if } k = 0, \\
    \sqrt{ \frac{c_{k} \cdots c_{1} }{b_{k-1}\cdots b_0}}, \text{ else}.
\end{cases}
\]
We can now consider the matrix $T_{2n} = P^{-1} M_{2n} P $. We have
\[
T_{2n} = \begin{pmatrix}
    a_0 & -\sqrt{b_0 c_1}  \\
    -\sqrt{b_0 c_1} & a_1 & -\sqrt{b_1 c_2}   \\
     & \ddots & \ddots & \ddots  \\ 
     &&&& -\sqrt{b_{n-1}c_n} \\
     & && -\sqrt{b_{n-1}c_n} & a_n
\end{pmatrix},
\]
which is indeed a symmetric tri-diagonal matrix. The spectral theorem implies that \(T_{2n}\) has real eigenvalues. Since, \(M_{2n}\) is similar to \(T_{2n}\), the two matrices have the same spectrum, hence $M_{2n}$ has real eigenvalues for all \(n\).
\end{proof}

Before continuing, we computed the smallest eigenvalue of $T_{2n}$ as a function of $n$ and presented the result in the following figure. We consider the case \(\lambda = 0\), by setting \(\delta = \lambda^{-1/2}\) and taking the limit as \(\lambda \to 0^+\) of the coefficient in the matrix \(T_{2n}\). This limit can be rigorously justified using arguments presented in the proof of Lemma \ref{lemma gap M_n}.
\begin{figure}[H]
\centering
        \includegraphics[width=0.49\linewidth]{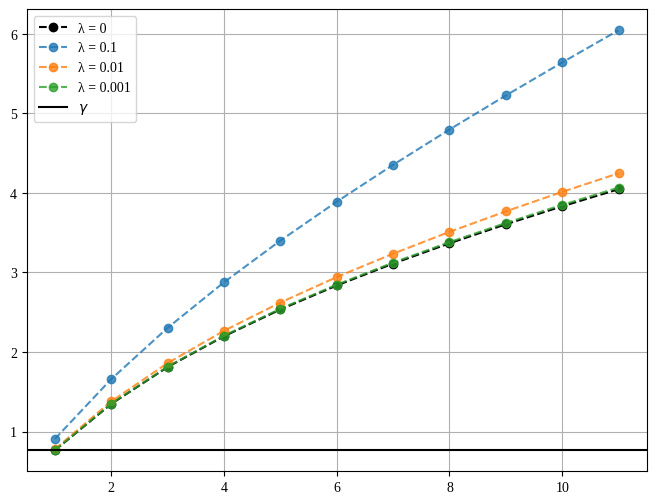}
        \includegraphics[width=0.49\linewidth]{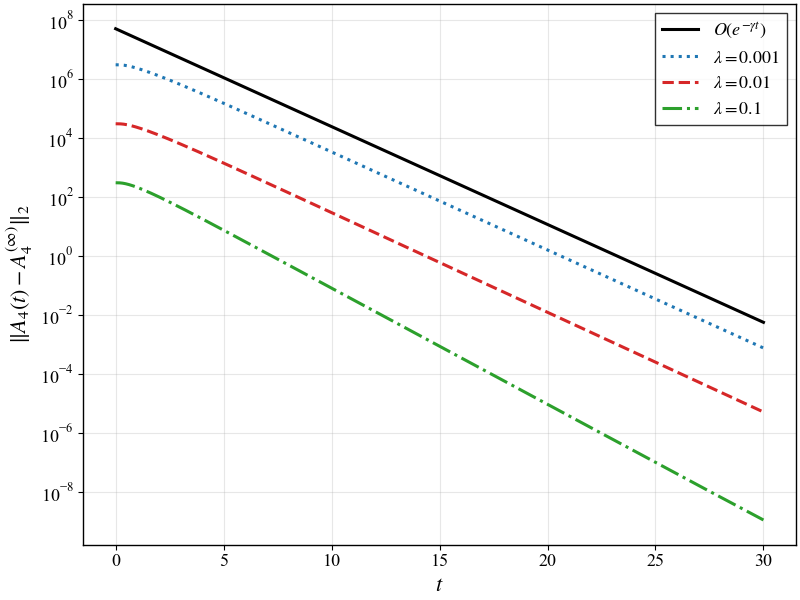}
        \caption{Left: Smallest eigenvalue of the matrix \(T_{2n}\) as function of \(n\). \newline
        Right: Dynamics \eqref{ODEs Moments} for different choices of \(\lambda\) with the critical scaling \(\delta = \lambda^{-1/2}\).}
\end{figure}

We observe that the spectral gap does not to converge to zero when $\lambda \to 0$ and that the convergence rate seems to increase when $n \to \infty$. The following result confirms that the convergence's rate does not vanish as \(\lambda \to 0^+, \text{ for all }~n~\geq~1\).

\begin{lemma} \label{lemma gap M_n}
Set $\delta = \lambda^{-1/2}$, then for all $\lambda >0,\: n \geq 1$ we have
\[
\Lambda(M_{2n}) \ge \inf_{\lambda>0} \, \Lambda(M_2) =  \gamma.
\]
\end{lemma}

\begin{proof}[Proof of Lemma \ref{lemma gap M_n}]
    We define
\begin{enumerate}
    \item[] \[ \alpha_k^{(n)} \coloneqq  a_k = 4n + 2k(4(n-k)+ \lambda-1), \: k = 0,\dots,n \]
    \item[] \[ \gamma_k^{(n)} \coloneqq \sqrt{b_k c_{k+1}} = 2\sqrt{(n-k)(2n-2k-1)(k+1)(2k+1)}, \: k = 0,\dots,n . \]
\end{enumerate}
Since $T_{2n}$ is a Jacobi matrix, its smallest eigenvalue is a non-decreasing function of the diagonal coefficients (see section II part 9 in \cite{Oscillation}). Thus, to prove the lemma, it suffices to consider the case \( \lambda = 0 \).

We aim to apply the lower bound from Gershgorin’s circle theorem. Let us define
\[ 
\eta_n(k) \coloneqq \left(1+ \frac{1}{n}\right) \sqrt{\frac{2(n-k)}{2(n-k)-1} },
\] 
and consider the matrix
\[
P\coloneqq \text{diag}(\delta_0, \dots,\delta_n), \text{ with }
\delta_k = \begin{cases}
    1, \text { if } k = 0, \\
    \prod_{j = 0}^{k-1} \eta_n(j), \text{ else},
\end{cases}
\]
such that the following change of basis gives
    \[P^{-1} (T_{2n})P = \begin{pmatrix} \alpha_0^{(n)} & \gamma_0^{(n)}\eta_n(0) & \\
    \gamma_0^{(n)}/\eta_n(0) & \alpha_1^{(n)} &  \gamma_0^{(n)}\eta_n(1) \\
    \quad \ddots & \quad\ddots & \quad\ddots \\ &&& \gamma_{n-1}^{(n)}\eta_n(n-1) \\
    &&\gamma_{n-1}^{(n)}/\eta_n(n-1) & \alpha_n^{(n)}
    \end{pmatrix}. \]
With this rescaling, we now apply the lower bound from the Gershgorin's circle theorem. Let us define
\[
G^{(n)}(k) \coloneqq \alpha_k^{(n)} - \left( \frac{\gamma_{k-1}^{(n)}}{\eta_n(k-1)}  +\gamma_k^{(n)} \eta_n(k)\right).
\]
Here we set $ \gamma_{-1}^{(n)} = 0$. We show in Lemma \ref{borne n ge 11} found in the Appendix \ref{Complementary results for the 2 dimensional cases} that for all \(n\geq 11\) and for all $k = 0,\dots,n$, $G^{(n)}(k) \geq 5/2$. In particular, for all \(n \geq 11\)
\[
\Lambda(M_{2n}) = \Lambda(P^{-1} (T_{2n})P) \geq \min_{n \ge 11, \:k \in [0,n]} G_k^{(n)} \ge \frac 5 2 \ge \gamma.
\]
For \( 1 \le n \le 11\), we consider
\[ 
\eta_n(k) \coloneqq \left(1+ \frac{1}{10}\right) \sqrt{\frac{2(n-k)}{2(n-k)-1} },
\] 
and by defining a similar transformation, we show in Lemma \ref{borne n leq 11} that the following holds for all \(2 \leq n \le 11\)
\[
\Lambda(M_{2n}) = \Lambda(P^{-1} (T_{2n})P)  \geq \min_{\substack{2\le n \le 11, \\ \:k \in [0,n]}}{G_k^{(n)}} \geq \frac{9}{11} \geq \gamma.
\]
We can finally conclude using Lemma \ref{Smallest eigenvalue M} that for all \(n \geq 1\)
\[
\Lambda(M_{2n})  \ge \inf_{\lambda >0} \,\Lambda(M_2) = \gamma.
\]
\end{proof}
\begin{remark}
The proof of Lemma \ref{lemma gap M_n} relies on the choice of the rescaling parameter \(\eta\). Indeed, applying the theorem directly to \( T_{2n} \) yields a negative value, which is useless, since we know that \( T_{2n} \) is positive definite. Different choices of \(\eta\) are possible. For instance, we initially considered
\[ 
\eta_n(k) \coloneqq \sqrt{\frac{2(n-k)}{2(n-k)-1} },
\] 
which leads to the weaker result \(\Lambda(M_{2n}) \ge\dfrac{1}{2}\). Moreover, we can see numerically that with the choice
\[ 
\eta_n(k) \coloneqq \left(1+ \frac{\log(n)}{n}\right) \sqrt{\frac{2(n-k)}{2(n-k)-1} },
\] 
we could show that \( \Lambda(M_{2n}) \) is bounded below by an increasing function of \(n\). However, the choice of the proof allows simpler computations and is strong enough to show that the second moment is the slowest of all.
\end{remark}

We conclude this section by proving the result from Theorem \ref{Main Result 2d} on polynomials.

\begin{proposition}
    
\label{SP Gap on polynomials}
    The spectral gap of \(\mathcal{R}\) restricted to the space of polynomial \(\mathcal{P}\) is at least \(\gamma\). In particular, we have
    \[
        \lVert e^{t\mathcal{R}}P - \pi(P) \rVert_{L^2(\pi)} \leq \lVert P - \pi(P) \rVert_{L^2(\pi)} e^{-\gamma t}, \text{ for all } P \in \mathcal{P}.
    \]
\end{proposition} 
\begin{proof}[Proof of Proposition \ref{SP Gap on polynomials}]
Using Lemma \ref{lemma gap M_n}, Remark \ref{odd moments not treated} and Proposition \ref{solution of ODEs Moments}, the spectral gap of \(\mathcal{R}\) restricted to the space of polynomial \(\mathcal{P}\) is at least \(\gamma\). We now show that this implies the exponential decay in the statement of the proposition. 

Let \(P \in \mathcal{P}\). If \(P\) is a constant, the bound is trivial. Let \(P \in \mathcal{P}\) be a non-constant function and define
    \[ U(t) = \lVert u(t) \rVert_{L^2(\pi)}^2 \coloneqq \lVert e^{t\mathcal{R}}P - \pi(P) \rVert_{L^2(\pi)}^2.\]
    To prove the result, we bound the derivative of \(U\) and use the Grönwall lemma. We can see that
    \[
    \frac{d}{dt}U(t) =2\inner{\frac{d}{dt}u}{u}
    \]
    By Kolmogorov's equations, we have
    \[
    2\inner{\frac{d}{dt}u}{u} =2\inner{\mathcal{R}u}{u}.
    \]
    Note that, since \(P\) is not a constant, \(u \neq 0\). Since the operator is symmetric and negative definite, we can use the fact that the Rayleigh quotient of \(u\in \mathcal{P}_0\) gives a bound on the spectral gap. For instance,
    \[
    \frac{\inner{\mathcal{R}u}{u}}{\inner{u}{u}} \leq -\gamma.   \]
    So we find
    \[
    2\inner{\mathcal{R}u}{u} \leq -2\gamma \inner{u}{u} = -2\gamma U(t).
    \]
    By the Grönwall lemma, we have
    \[
    U(t) \leq U(0)e^{-2\lambda t}.
    \]
    And we can conclude, since taking the square root on both sides yields
    \[
    \lVert e^{t\mathcal{R}}P - \pi(P) \rVert_{L^2(\pi)} \leq \lVert P - \pi(P) \rVert_{L^2(\pi)} e^{-\gamma t}.
    \]
\end{proof}

\subsection{Extension to general observables} \label{Extension to general observables}
Consider the generator arising from the Stratonovich perturbation
\[
\mathcal{R} = \mathcal{L} + \mathcal{A}^2.
\]
It is a reversible, self-adjoint in \( L^2(\pi)\) and negative semi-definite operator. What we have shown in the previous sections that, for all polynomials \(P \in \mathcal{P} \), we have
\begin{equation}
    \lVert e^{t\mathcal{R}}P - \pi(P) \rVert_{L^2(\pi)} \leq \lVert P - \pi(P) \rVert_{L^2(\pi)} e^{-\gamma t}. \label{Bound polynomial}
\end{equation}
We want to extend this result to general observable and conclude that the spectral gap is at least~\( \gamma\). We will use the following lemmas.

\begin{lemma} \label{Constant}
Let \(\lambda \) be the spectral gap of \( \mathcal{R}\) in \(L^2(\pi)\), a symmetric negative semi-definite operator, then, for all \(\phi \in L^2(\pi) \), we have
    \[
    \lVert e^{t\mathcal{R}}\phi - \pi(\phi) \rVert_{L^2(\pi)} \leq  \lVert \phi - \pi(\phi) \rVert_{L^2(\pi)}  e^{-\lambda t}. 
    \] 
    
\end{lemma}

\begin{proof}[Proof of Lemma \ref{Constant}]
The proof follows the line of Proposition \ref{SP Gap on polynomials}, but we state it for completeness.

    If \(\phi\) is a constant, the bound is trivial. Let \(\phi \in L^2(\pi)\) be a non-constant function and define
    \[ U(t) = \lVert u(t) \rVert_{L^2(\pi)}^2 \coloneqq \lVert e^{t\mathcal{R}}\phi - \pi(\phi) \rVert_{L^2(\pi)}^2.\]
    To prove the result, we bound the derivative of \(U\) and use the Grönwall lemma. We can see that
    \[
    \frac{d}{dt}U(t) =2\inner{\frac{d}{dt}u}{u}
    \]
    By Kolmogorov's equations, we have
    \[
    2\inner{\frac{d}{dt}u}{u} =2\inner{\mathcal{R}u}{u}.
    \]
    Note that, since \(\phi\) is not a constant, \(u \neq 0\). Since the operator is symmetric and negative definite, we can use the fact that the Rayleigh quotient of \(u\in L^2_0(\pi)\) gives a bound on the spectral gap. For instance,
    \[
    \frac{\inner{\mathcal{R}u}{u}}{\inner{u}{u}} \leq -\lambda.   \]
    So we find
    \[
    2\inner{\mathcal{R}u}{u} \leq -2\lambda \inner{u}{u} = -2\lambda U(t).
    \]
    By the Grönwall lemma, we have
    \[
    U(t) \leq U(0)e^{-2\lambda t}.
    \]
    And we can conclude, since taking the square root on both sides yields
    \[
    \lVert e^{t\mathcal{R}}\phi - \pi(\phi) \rVert_{L^2(\pi)} \leq \lVert \phi - \pi(\phi) \rVert_{L^2(\pi)} e^{-\lambda t}.
    \]
\end{proof}

\begin{lemma} \label{Bound spectral gap}
    Suppose that
    \[
    \lVert e^{t\mathcal{R}}\phi - \pi(\phi) \rVert_{L^2(\pi)} \leq  \lVert \phi - \pi(\phi) \rVert_{L^2(\pi)} e^{-\gamma t}, \text{ for all } \phi \in L^2(\pi).
    \]
    Then the spectral gap \(\lambda\) of \( \mathcal{R}\) is at least \(\gamma\).
\end{lemma} 
\begin{proof}[Proof of Lemma \ref{Bound spectral gap}]
    Let \(\psi\) be the eigenvector corresponding to \(-\lambda\) in an ONB. We have
    \[
    e^{t\mathcal{R}} \psi = e^{-\lambda t} \psi, \: \pi(\psi) = \inner{\psi}{\mathbf{1}}_\pi = 0.
    \]
    We then have for all \(t\)
    \[
    e^{-\lambda t} =\lVert e^{t\mathcal{R}}\psi - \pi(\psi) \rVert_{L^2(\pi)} \leq \lVert \psi - \pi(\psi) \rVert_{L^2(\pi)} e^{-\gamma t} = e^{-\gamma t}. 
    \]
    This is possible if and only if \(\lambda \ge \gamma\)
\end{proof}

\begin{proof}[Proof of Theorem \ref{Main Result 2d}]
    
From Proposition \ref{SP Gap on polynomials}, we have the bound
    \[
        \lVert e^{t\mathcal{R}}P - \pi(P) \rVert_{L^2(\pi)} \leq \lVert P - \pi(P) \rVert_{L^2(\pi)} e^{-\gamma t}.
    \]
Since \( \mathcal{R} \) is a semi-definite negative operator, \(e^{t \mathcal{R}}\) is a bounded linear operator, so it is continuous. Moreover, it is known that the polynomials are dense in \(L^2(\pi)\). Using these both arguments, we can then extend the bound to a general observable \(\phi \in L^2(\pi)\). This leads to
\begin{equation}
        \lVert e^{t\mathcal{R}}\phi - \pi(\phi) \rVert_{L^2(\pi)} \leq  \lVert \phi - \pi(\phi) \rVert_{L^2(\pi)}  e^{-\gamma t}, \text{ for all } \phi \in L^2(\pi). 
\end{equation}
By Lemma \ref{Bound spectral gap}, this implies that the spectral gap \(\lambda\) is at least \(\gamma\). This finishes the proof of Theorem~\ref{Main Result 2d}.
\end{proof}

\section{Proof in general dimension: Rayleigh estimation}\label{Proof in a general dimension: Rayleigh estimation}
In this section, we consider the same setting as in Section \ref{Pairwise matching in a general dimension}. In order to prove Theorem \ref{Main result d=2n asymptotic}, we find a lower bound for the Rayleigh quotient. Indeed, for reversible dynamics, the spectral gap \(\lambda\) of the associated generator \(\mathcal{R}\) is given by the Rayleigh quotient
\[
\lambda= \inf_{ u \in L^2_0(\pi)} \frac{\langle u,-\mathcal{R}u\rangle_{L^2(\pi)}} {\langle u,u\rangle_{L^2(\pi)}} = \inf_{\substack{ u \in L^2_0(\pi) \\ \langle u,u\rangle = 1 }} {\langle u,-\mathcal{R}u\rangle_{L^2(\pi)} }.
\]
    In our case, \(\mathcal{R} = \mathcal{L} + \mathcal{A}_g^2\). Since \(\inf A + B \geq \inf A + \inf B\), we have the lower bound
    \[
    \lambda \geq
    \underbrace{\inf_{\substack{ u \in L^2_0(\pi) \\ \langle u,u\rangle = 1 }} {\langle  u,-\mathcal{L} u\rangle}}_{A} + \underbrace{\inf_{\substack{ u \in L^2_0(\pi) \\ \langle u,u\rangle = 1 }} {\langle u,-\mathcal{A}_g^2 u\rangle}}_{B}.
    \]
In the quadratic case, \((A)\) is well understood and corresponds to the unperturbed Langevin dynamics \eqref{OLD}, so the main effort is to understand \((B)\). Note that looking for the infimum on the whole space leads to unsatisfying result \(B = 0\).

In our analysis, we start by reducing the problem to a family of finite dimensional space \((\mathcal{H}_k)_{k\geq1}\). Then, in order to obtain a tangible bound, we estimate the Rayleigh quotient on the kernel \(A^k = \ker(\mathcal{A}) \cap \mathcal{H}_k\) and on its orthogonal complement \(B^k = (A^k)^\perp = \{ p \in \mathcal{H}_k : p \perp a, \forall a \in A^k \}\). I.e, we estimate
        \[
        \inf_{\substack{ u \in A^k \\ \langle u,u\rangle = 1 }} {\langle u,-\mathcal{R}u\rangle_{L^2(\pi)} }, \quad \text{and} \quad     \inf_{\substack{ u \in B^k \\ \langle u,u\rangle = 1 }} {\langle u,-\mathcal{R}u\rangle_{L^2(\pi)} }.
        \]
We then use these estimates to find a lower bound of the Rayleigh quotient restricted to the space \(\mathcal{P}_0\) of polynomials orthogonal to constant. Finally, we use a density argument to extend the result to the whole space as in Section \ref{Extension to general observables}.

In this section, we use the following notation for the Rayleigh quotient
    \[
    R_\delta(f) = \frac{\inner{f}{ -\mathcal{R}_\delta) f}}{\inner{f}{f}} = \frac{\inner{f}{(\nabla^\# \nabla + \delta^2 \mathcal{A}^\# \mathcal{A}) f}}{\inner{f}{f}},
    \]
where \(\mathcal{A} = JSx \cdot \nabla\), with \(J\) of the form \eqref{J pairwise matching}. Moreover, when \(\delta = 1\), we note \(R_1(f) = R(f)\).We also denote \(\delta\) instead of \(\hat \delta\), because we only used this notation in Section \ref{Pairwise matching in a general dimension} to avoid confusion with the scaling in two dimensional scaling from Section \ref{2D Results}. 

\subsection{Reduction to finite dimensional space}\label{section reduction to Wiener Chaos}
In order to reduce the problem to a finite dimensional space, we present a decomposition using the so called \(k\)-th Wiener chaos \(\mathcal{H}_k\). Despite their name, they are not related to chaos theory, but they were introduced by Wiener to decompose a Gaussian random variable into a sum of polynomials. See for instance \cite[Chapter~2]{Janson_1997} for a precise introduction. We now recall how these spaces are defined. Let us consider the one dimensional Hermite polynomials
\[
h_n(x) = (-1)^n e^{x^2/2}  \partial_x^n e^{-x^2/2},
\]
and the scaled and normalized \(d\)-dimensional version
\begin{equation}
    H_\alpha(\textbf{x}) =\prod_{i=1}^d \frac{1}{\sqrt{\alpha_i!}} h_{\alpha_i}(\sqrt{s_i} \textbf{x}_i), \quad  \alpha \in \mathbb{N}^d. \label{definition d-dimensional hermite}
\end{equation}
It appears that the polynomials \eqref{definition d-dimensional hermite} are eigenvectors of the operator \(\nabla^\# \nabla\) and they form an ONB in \(L^2(\pi)\). In particular, let us define the \(k\)-th Wiener chaos \(\mathcal{H}_k\) as
    \[
    \mathcal{H}_k = \text{ span}\{ H_\alpha \big| \:   \alpha \in \mathbb{N}^d, |\alpha| =\sum_{i=1}^d \alpha_i = k\}.
    \]
We have
\[
L^2_0(\pi) = \bigoplus_{k\geq 1} \mathcal{H}_k.
\]

\begin{lemma} \label{Wiener Chaos preservation}
    The generator \(\mathcal{A}\) preserves the \(n\)-th Wiener chaos for all \(n\in \mathbb{N}\). That is
    \[
    \mathcal{A}( \mathcal{H}_n) \subset \mathcal{H}_n.
    \]
    Therefore, \(\mathcal{R}(\mathcal{H}_m) \subset \mathcal{H}_m \).
\end{lemma}
\begin{proof}
    We consider the change of variable \(S^{1/2} \mathbf x = \mathbf y\). The operator \(\mathcal{A
    }\) becomes
    \[
    \mathcal{A} =\sum_{i=1}^n \mathbf y_{d+1-i} \partial_i - \mathbf y_{i} \partial_{d+1-i} = \sum_{i=1}^n a_{d+1-i}^\#a_i - a_{i}^\# a_{d+1-i},
    \]
    where we introduced \(a_i = \partial_i\) and \(a_i^\# = \mathbf y_i -\partial_i\). The Hermite polynomials satisfy
    \[
    a_i H_{\alpha} = \alpha_i H_{\alpha - e_i},
    \]
    where \(e_i\) is the \(i\)-th canonical vector in \(\mathbb{R}^d\). Moreover,
    \[
    a_i^\# H_{\alpha} = H_{\alpha + e_i}.
    \]
    In particular, the action of any pair \( a_{i}^\# a_{d+1-i}\) preserves the total degree. Hence so do \(\mathcal{A}\). Therefore, for any \(f \in \mathcal{H}_n\), we have \(\mathcal{A}f_n \in \mathcal{H}_n\). Finally, Hermite polynomials are eigenfunctions of the operator \(\mathcal{L}\). Hence, \(\mathcal{R}(\mathcal{H}_m) \subset \mathcal{H}_m \).
\end{proof}

\begin{proposition}\label{reduction to Wiener Chaos}
    For any \(f \in L^2_0(\pi)\backslash\{\mathbf{0}\}\) , we have
    \[
    R(f) \geq  \inf_{k \geq 1} \inf_{f_k \in \mathcal{H}_k} R(f_k).
    \]
\end{proposition}
\begin{proof}
    Since \(L^2_0(\pi)\) is a direct sum of \(\mathcal{H}_k\) for \(k\geq 1\), we can write any \(f \in L^2_0(\pi)\backslash\{\mathbf{0}\}\) as
    \[
    f = \sum_{k\geq 1} f_k, \quad f_k \in \mathcal{H}_k.
    \]
    Since the family \(( \mathcal{H}_k)_{k\ge1}\) is pairwise orthogonal and using Lemma \ref{Wiener Chaos preservation}, we obtain
    \[
    R(f) = \frac{\sum_{k \ge 1} \inner{f_k }{\mathcal{R}f_k}}{\sum_{k \ge 1} \inner{f_k }{f_k}} \geq \inf_{k \geq 1} \inf_{f_k \in \mathcal{H}_k} R(f_k).
    \]
\end{proof}

\subsection{Characterization of the kernel \texorpdfstring{$A^{k}$}{Ak}} \label{Characterization of the kernel\texorpdfstring{$A^{k}$}{Ak}}
\begin{proposition}[Characterization of $A^k$] \label{Characterization of $A^k$}
    The set
    \[
    A^{k} = \left\{ u\in\mathcal{H}_{k} | \: \forall \textbf x \in \mathbb{R}^d ,\,\nabla u(\textbf x) ^TJS \textbf x = 0 \right\},
    \]
    contains only polynomials with monomials of even degree. In fact, it is spanned by polynomials of the form 
    \[
    \Phi = \Phi(\mathcal{R}e(z_i\bar z_j),\mathcal{I}m(z_i\bar z_j)), \: \Phi \in \mathcal{P}^{k},
    \]
    where \(z_i = \sqrt{s_i}x_i +\mathrm{i}\sqrt{s_{d+1-i}}x_{d+1-i}\).
\end{proposition}
\begin{proof}
    We want to find all solutions of \(\nabla u(\textbf x)^TJS \textbf x = 0 \). Written algebraically, we have
    \[
    \sum_{i=1}^d \left(JS \textbf x\right)_i \partial_{x_i}u( \textbf x) = 0.
    \]
    The characteristic of this PDE is the set
    \[
    \left \{ \begin{pmatrix}
        X_1(t)&\dots & X_d(t), c
    \end{pmatrix} | u(\mathbf{X} (t)) =c, \dot{\mathbf{X}}(t) = JS \mathbf{X} (t) \right\}.
    \]
    In our case, \((JS \textbf x)_i = (-1)^{\lfloor i/(n+1) \rfloor} a_{i} s_{d+1-i}x_{d+1-i}\), hence the ODE is in fact \(n\) system of coupled two variables ODE 
    \[
    \begin{cases}
        \dot X_i = a_i s_{d+1-i}X_{d+1-i}, \\
        \dot X_{d+1-i}  = -a_{i} s_{i}X_i.
    \end{cases}
    \]
    By introducing the variable \(z_i\), the system becomes
    \[
    \dot z_i = - \mathrm{i}a_i\sqrt{s_i s_{d+1-i}} \:z_i.
    \]
    With our definition of \(J\) \(a_i\sqrt{s_i s_{d+1-i}}=1 \), so we can write the whole system in \(\mathbb C^n\) as
    \[
    \dot {\mathbf z} = -i \mathbf z.
    \]
    The polynomials functions that are conserved quantities of this system are polynomials in variables \(z_i \bar z_j\). Indeed, let us denote \(z_j = r_j e^{i \theta_j}\). If \(z^{\alpha} \bar z^{\beta}\) is a monomial and a conserved quantity, we must have
    \[
    \mathbf z^{\alpha} \bar {\mathbf z}^{\beta}  = \prod_{j=1}^n r_j^{\alpha_j + \beta_j} e^{i \theta_j (\alpha_j - \beta_j)} = \prod_{j=1}^n r_j^{\alpha_j + \beta_j} e^{i \theta_j (\alpha_j - \beta_j) + it\sum_{j=1}^n\alpha_j - \beta_j}  = \mathbf z^{\alpha}_t \bar {\mathbf z}^{\beta}_t.
    \]
    For that to be true, we must have \(|\alpha| = |\beta|\). Thus, we can decompose monomials that are conserved quantities as product of \(z_j \bar z_k\). For this polynomial to be a conserved quantity, it must satisfy \(P(z_t,\bar z_t)=P(z,\bar z) \) for all \(t\). Substituting the evolution of the system, we get
    \[
    \sum_{\alpha,\beta} c_{\alpha,\beta} z^{\alpha} \bar z^\beta e^{it(|\alpha|-|\beta|)} = \sum_{\alpha,\beta} c_{\alpha,\beta} z^{\alpha} \bar z^\beta.
    \]
    By independence of the monomials, it is possible if only if \(|\alpha| = |\beta|\) for all \(\alpha,\beta\) contributing in the sum so we can write all contributing monomials as product of \(z_j \bar z_k\).
    
    The argument of the method of characteristics is to use the fact that \(u\) is a solution of the PDE if and only if it is a conserved quantities of the ODE. Indeed, if \(\Phi(\mathbf X_t) =\Phi(\mathbf X_0)\) is a conserved quantity of the ODE, it solves the PDE since
    \[
    0 = \dot \Phi(\mathbf X) = \nabla \Phi(\mathbf X)^T\dot X = \nabla\Phi(\mathbf X)^TJS \mathbf X.
    \]
     In particular, when \(\Phi\) is a polynomial of degree \(k\), it belongs to \(A^k\). Moreover, if \(u\) is a solution of \(\nabla u(\mathbf X)^TJS \mathbf X = 0 \), then \(u(\mathbf X(t))\) is a conserved quantity since
     \[
     \frac{d}{dt} u(\mathbf X(t)) = \nabla u(\mathbf X(t))^T JS \mathbf X(t) = 0.
     \]
     Hence, \(u(\mathbf X(t)) = u(\mathbf X_0)\). We can write for all \(u\in A^{2k}\)
    \[
    u(\mathbf x) = u \left( \mathcal{R}e(z_i \bar z_j), \mathcal{I}m(z_i \bar z_j), i,j = 1,\dots,n \right).
    \]
    We note that the conserved quantities are polynomials of degree \(2\). Therefore, \(u\) must be a polynomial of even order.

\end{proof}
\begin{remark} \label{R_theta for A2k}
    Let us consider the change of variables \(\mathbf y = S^{1/2}\mathbf x\). We can see from the proof that for all \(u(\mathbf y) \in A^{2k}\), we have
    \[
    R^{(\theta)} \, u(\mathbf y) = u(\mathbf y),
    \]
    Where \(R^{(\theta)}\) is an operator defined as
    \[
    R^{(\theta)} \, u(\mathbf y)\coloneqq u(M_\theta \,\mathbf y),\quad M_\theta = \cos(\theta) I_d + \sin(\theta) J_d,
    \]
    where \(J_d\) is the symplectic skew-symmetric matrix. Moreover, we can note that the adjoint of \(R^{(\theta)}\) in \(L^2(\pi(dy))\) is simply \(R^{(-\theta)}\). This follows directly from a change of variable and the fact that \(M_\theta\) is orthogonal.
\end{remark}

    \subsection{Spectral gap on \texorpdfstring{$A^{k}$}{Ak}} \label{Spectral gap on A2k}
    Let \(k>1\). Our goal is to find a lower bound for
    \[
    \inf_{u\in A^{k}} R(u),
    \]
    where
    \[
     A^k = \ker(\mathcal{A}_{|\mathcal{H}^{k}}),\: \mathcal{A}_g= -x^TSJ \nabla, \: R(u) = \frac{\langle \nabla u,\nabla u\rangle}{\langle u,u\rangle}.
    \]
    Recall that from the previous section, \(A^{2k+1} = \{0\}\) for all \(k\ge0\). Therefore, we focus our analysis on even order.

    \begin{proposition}\label{Prop sp gap A^2k}
        For all \(k\ge1\), we have
        \[
        \inf_{u  \in A^{2k} \backslash \{0\}} R(u) \geq  k \gamma_{\min},
        \]
        with equality when \(k = 1\).
    \end{proposition}
    \begin{proof}
    We consider the change of variable \(S^{1/2} \mathbf x = \mathbf y\). The measure \(\pi(d\mathbf x) \sim \mathcal{N}(0, S^{-1})\) becomes \(\pi(d\mathbf y) \sim\mathcal{N}(0, I)\). Moreover, the part of the generator acting on \(x_i\), which we denote \(\mathcal{L}_i^{(\mathbf x)}\), becomes
    \[
    \mathcal{L}_i^{(\mathbf x)} = -s_i x_i \partial_{x_i} + \partial_{x_i}^2  = s_i( -y_i \partial_{y_i} + \partial_{y_i}^2) = -s_i E_i.
    \]
    We have
    \begin{equation}
            \langle \nabla_\mathbf x u, \nabla_ \mathbf x u \rangle = -\sum_{i=1}^d \langle u, \mathcal{L}_i^{(\mathbf x)} u \rangle = \sum_{i=1}^d s_i\langle u, E_i u \rangle. \label{change of var in sp gap Ak}
    \end{equation}

    Since \( u \in A^{2k}\) and using the remark \ref{R_theta for A2k} we see that, for all \(i=1,\dots,d\) and for all \(\theta \in [0,2\pi)\), we have
    \[
    \langle u, E_i u \rangle = \langle R^{(\theta)} u, E_i R^{(\theta)} u \rangle = \langle  u, R_{-\theta} E_i R^{(\theta)} u \rangle.
    \]
    Using Fubini-Tonelli, we have
    \[
    \langle u, E_i u \rangle  = \frac{1}{2\pi} \int_0^{2\pi} \langle  u, R_{-\theta} E_i R^{(\theta)} u \rangle \, d\theta = \langle u, \bar E_i u \rangle,
    \]
    where 
    \[
    \bar E_i = \frac{1}{2\pi} \int_0^{2\pi}  R_{-\theta} E_i R^{(\theta)}  \, d\theta.
    \]
    We can compute explicitly the averaged operator (see Lemma \ref{Computation Ei bar}) 
    \[
    \bar E_i = \frac{E_i + E_{d+1-i}}{2}.
    \]
    Therefore, 
    \[
     \langle \nabla_x u, \nabla_x u \rangle = \sum_{i=1}^d \frac{s_i}{2}\langle u, (E_i+ E_{d+1-i}) u \rangle = \sum_{i=1}^d \frac{s_i + s_{d+1-i}}{2}\langle u, E_i u \rangle.
    \]
    A natural lower bound is 
    \[
   \langle \nabla_x u, \nabla_x u \rangle \geq \min_{i=1,\dots,n}\frac{s_i + s_{d+1-i}}{2} \sum_{j=1}^d \langle u, E_j u \rangle = \min_{i=1,\dots,n}\frac{s_i + s_{d+1-i}}{2} \langle \nabla_y u, \nabla_y u \rangle.
    \]
    Finally, we obtain
    \[
   R(u) = \frac{\langle \nabla_x u, \nabla_x u \rangle}{\langle u, u \rangle} \geq \min_{i=1,\dots,n}\frac{s_i + s_{d+1-i}}{2} \inf_{\substack{u \in \mathcal{H}_{2k} \backslash \{0\} \\ \text{deg}(u) \ge 2} } \frac{\langle \nabla_y u, \nabla_y u \rangle}{\langle u, u \rangle} = \min_{i=1,\dots,n}\frac{s_i + s_{d+1-i}}{2} \: 2k.
    \]
    Where the last equality is a known fact about the Ornstein-Uhlenbeck generator with identity for its drift matrix. 
    
    Finally, consider the function \(u_{i^{*}i^{*}}(\textbf{x}) = |z_{i^*}|^2\), with \(i^*\) the index such that \(\gamma_{\min} = s_{i^*}+ s_{d+1-i^*}\). A straightforward calculation for the function \(f = u_{i^{*}i^{*}} - 2 \in A^{2} \backslash \{0\}\) gives
    \[
    R(f) =  \min_{i=1,\dots,n}{s_i + s_{d+1-i}}.
    \]
    \end{proof}
    \begin{remark} We note that the naive bound
    \[
     \langle \nabla_\mathbf x u, \nabla_\mathbf x u \rangle \geq \min_{i=1,\dots,n} s_i \langle \nabla_\mathbf y u, \nabla_\mathbf y u \rangle,
    \]
    applied to \eqref{change of var in sp gap Ak} does not give anything new since we recover the known bound of the Ornstein-Uhlenbeck generator corresponding to the case \(J=0\). We should therefore use the structure of \(A^{k}\) if we want to gain something.
    \end{remark}

 \begin{lemma} \label{Computation Ei bar}
        In the notation of the proof above, we have
    \[
    \bar E_i = \frac{E_i + E_{d+1-i}}{2}.
    \]
    \end{lemma}
    \begin{proof}
        We notice that the rotation induced from \(R_{\theta}\) exactly adds \(\theta\) to each angles between \(y_i\) and \(y_{d+1-i}\). For each pair \((y_i, y_{d+1-i})\) the following changes appears
        \[
        \begin{cases}
            y_i \mapsto y_i \cos(\theta) + y_{d+1-i} \sin(\theta), \\
            y_{d+1-i} \mapsto -y_i \sin(\theta) + y_{d+1-i} \cos(\theta).
        \end{cases}
        \]
        By linearity of the derivative
                \[
        \begin{cases}
            \partial_{y_i} \mapsto \partial_{y_i} \cos(\theta) + \partial_{y_{d+1-i}} \sin(\theta), \\
            \partial_{y_{d+1-i}} \mapsto -\partial_{y_i} \sin(\theta) + \partial_{y_{d+1-i}} \cos(\theta).
        \end{cases}
        \]
        Therefore,
        \[
        y_i \partial_{y_i} \mapsto \cos(\theta)^2 y_i \partial_{y_i} + y_i \cos(\theta)\sin(\theta) \partial_{y_{d+1-i}} + y_{d+1-i} \cos(\theta)\sin(\theta) \partial_{y_{i}} + \sin(\theta)^2 y_{d+1-i} \partial_{y_{d+1-i}}.
        \]
 Similarly
        \[
        \partial_{y_i}^2 \mapsto \cos(\theta)^2 \partial{y_i}^2 + 2 \cos(\theta)\sin(\theta) \partial_{y_i} \partial_{y_{d+1-i}} + \sin{\theta}^2 \partial_{y_{d+1-i}}^2. 
        \]
        We use standard formulas for integrals of trigonometric functions to obtain
        \[
        \bar E_i = \frac{1}{2} \left(y_i \partial_{y_i} + y_{d+1-i} \partial_{y_{d+1-i}} -\partial{y_i}^2 - \partial_{y_{d+1-i}}^2  \right)  = \frac{E_i + E_{d+1-i}}{2}.
        \]
    \end{proof}

We can extend the result to the whole kernel of \(\mathcal{A}\).

\begin{proposition} \label{sp gap on ker(A)}
We have
    \[
    \inf_{\substack{ g \in \ker(\mathcal{A}) \\ \langle g,g\rangle = 1, \:\langle g,1\rangle = 0 }} R(g) \geq \gamma_{\min}.
    \]
\end{proposition}
    \begin{proof}
        We first note that
        \[
        \ker(\mathcal{A}) \cap L^2_0(\pi) = \bigoplus_{k\ge 1} A_{2k}.
        \]
        Indeed, the inclusion \(\supseteq\) is trivial and for \(\subseteq\), let us consider \(g = \sum_{k} g_k \in \ker(\mathcal{A}) \cap L^2_0(\pi)\), with \(g_k \in \mathcal{H}_k\). By contradiction, we suppose that there exists \(g_l \in \mathcal{H}_l \backslash A_l\). Then \(\mathcal{A}g_l \neq 0\). On the other hand, due to the fact that \(\mathcal{A}(\mathcal{H}_k) \subset \mathcal{H}_k\) by Lemma \ref{Wiener Chaos preservation} and that \((\mathcal{H}_k)_{k\ge 1}\) is an orthogonal family, we have
        \[
        0 = \inner{\mathcal{A}g}{\mathcal{A}g} = \sum_{k\ge1} \inner{\mathcal{A}g_k}{\mathcal{A}g_k} \geq \inner{\mathcal{A}g_l}{\mathcal{A}g_l} >0,
        \]
        which is a contradiction. Therefore, \(g_k \in A^{k}\) for all \(k\ge 1\). Finally, we have \( A^{2k+1} = \{ 0\}\). So we find that
        \[
        \ker(\mathcal{A}) \cap L^2_0(\pi) = \bigoplus_{k\ge 1} A_{2k}.
        \]
        Moreover, by Lemma \ref{Wiener Chaos preservation} and Proposition \ref{reduction to Wiener Chaos}, we have
        \[
        \inf_{\substack{ g \in \ker(\mathcal{A}) \\ \langle g,g\rangle = 1, \:\langle g,1\rangle = 0 }} R(g) \geq \inf_{k\ge 1} \inf_{\substack{ g \in A^{2k} \\ \langle g,g\rangle = 1}} R(g).
        \]
        Finally, Proposition \ref{Prop sp gap A^2k} allows to lower bound the right hand side by \(\gamma_{\min}\) so that
        \[
         \inf_{\substack{ g \in \ker(\mathcal{A}) \\ \langle g,g\rangle = 1, \:\langle g,1\rangle = 0 }} R(g) \ge \gamma_{\min}.
        \]
    \end{proof}

\subsection{Spectral gap on \texorpdfstring{$B^k$}{Bk} } \label{Spectral gap on Bk}
    We first note that since \(\inf A + B \geq \inf A + \inf B\), we have the lower bound
    \[
    \inf_{u\in B^k \backslash\{0\}} R(u) \geq
    \inf_{u\in B^k \backslash\{0\} } \frac{\langle \nabla u,\nabla u\rangle}{\langle u,u\rangle} + \inf_{u\in B^k \backslash\{0\} } \frac{\langle g^T\nabla u,g^T\nabla u\rangle}{\langle u,u\rangle} 
    \geq k\,s_d + \inf_{u\in B^k \backslash\{0\} } \frac{\langle g^T\nabla u,g^T\nabla u\rangle}{\langle u,u\rangle}.
    \]
    We note that this lower bound is loose because the operators \(\mathcal{L}\) and  \(\mathcal{A}\) do not commute. In particular, they do not share the same eigenbasis and their Rayleigh quotient cannot be minimized simultaneously by the same eigenfunction. However, the approach we choose is simpler and yields a result sufficiently strong for our purpose.
    
    Our strategy will be to estimate \(\inf_{u\in B^k \backslash\{0\} } \frac{\langle g^T\nabla f,g^T\nabla f\rangle}{\langle f,f\rangle}\) by computing the spectrum of \(\mathcal{A}\) on \(B^k\).
    
\begin{proposition}
    
 \label{Spectrum of T}
Let \(\mathbb{K}\) be a field (either \(\mathbb{C}\) or \(\mathbb{R}\)). Let \(B \in \mathbb K^{d\times d}\) and let 
\[
\mu = \mathcal N(0,S^{-1}),
\qquad 
S=\mathrm{diag}(s_1,\dots,s_d), \; s_i>0.
\]
Consider the operator
\[
T = \mathbf x^\top B \nabla
\]
defined on the domain of polynomials in \(L^2(\mu)\). 

Let \(b_1,\dots,b_d\) denote the eigenvalues of \(B\), counted with algebraic multiplicity. Suppose that \(B\) is diagonalisable in \(\mathbb{K}\). Then the spectrum of \(T\) in the space of polynomials \(\mathcal{P}[\mathbb{K}]\) is
\[
\sigma(T)
=
\left\{
\sum_{i=1}^d \alpha_i b_i
\;:\;
\mathbf{\alpha} \in \mathbb N^d
\right\} \subset \mathbb{K}.
\]
\end{proposition}
\begin{proof}
    We first note that if \(\mathbf v_i\) is an eigenvector of \(B\) such that \(B\mathbf v_i = b_i \mathbf v_i\), then \(f_i(\mathbf x) = \mathbf v_i^T \mathbf x\) is an eigenvector of \(T\) with eigenvalue \(b_i\). Indeed, we have
    \[
    T f_i(\mathbf x) = T (\mathbf v_i^T \mathbf x) = \mathbf x^T B \mathbf v_i = b_i \mathbf x^T\mathbf v_i = b_i f_i(\mathbf x).
    \]
    We also note that \(T\) satisfy the product rule
    \[
    T(fg) = (Tf)g + f(Tg).
    \]
    In particular, if \(f,g\) are eigenfunction of \(T\) with eigenvalues \(\alpha,\beta\), we have that \(fg\) is an eigenfunction with eigenvalues \(\alpha + \beta\). Hence, any function of the form
    \[
    p_\mathbf a(\mathbf x) = \prod_{i=1}^d f_i(\mathbf x)^{a_i}, \tag{\(\star\)}
    \]
    is an eigenfunction of \(T\) with eigenvalue 
    \[
    \sum_{i=1}^d a_i b_i.
    \]
    Let us denote
    \[\mathcal{B}_k = \left\{ p_\mathbf a(\mathbf x)=\prod_{i=1}^d f_i(\mathbf x)^{a_i} , \text{ such that }  \mathbf a\in \mathbb{N}^d, |\mathbf a| \leq k  \right \}.
    \]
       
    Using Lemma \ref{B_k spans polynomials}, if a polynomial \(p(x)\) is not of the form (\(\star\)), we can write it as a linear combination of function of the form (\(\star\)). I.e.
    \[
    p(\mathbf x) = \sum_{|\mathbf \alpha|\leq k} c_\alpha p_\alpha(\mathbf x).
    \]
    Then, we can see that 
    \[
    Tp(\mathbf x) = \sum_{|\alpha|\leq k}  c_\alpha \left( \sum_{i=1}^d \alpha_i b_i \right) p_\alpha(\mathbf x).
    \]
    Therefore, if \(p(\mathbf x)\) is an eigenfunction of \(T\) with eigenvalue \(a\) and \(p(\mathbf x)\) is not of the form (\(\star\)), the following must hold for all \(\alpha \) indexing terms contributing in the linear combination (i.e \(c_\alpha \neq 0\))
    \[
    \left( \sum_{i=1}^d \alpha_i b_i \right) = a, 
    \]
    Therefore, the spectrum of \(T\) on polynomials is the following set
\[
\left\{
\sum_{i=1}^d \alpha_i b_i
\;:\;
\alpha \in \mathbb N^d
\right\}.
\]
    \end{proof}

\begin{lemma} \label{B_k spans polynomials}
        In the notation of Proposition \ref{Spectrum of T}, when \(B\) is diagonalisable, the set \(\mathcal{B}_k\) spans of polynomials of degree \(k\).
\end{lemma}
\begin{proof}
We proceed by induction.

    If \(k=1\), let \( f(\mathbf x) = \mathbf a^T\mathbf x\), with \(\mathbf a \in \mathbb{R}^d\). Since \(B\) is diagonalisable, there exists \(\mathbf v_i \in \mathbb{K}^d, \: i =1.\dots,d\) an eigenbasis of \(\mathbb{K}^d\). We have by linearity
    \[
    \mathbf a = \sum_{i=1}^d a_i \mathbf v_i \implies f(\mathbf x) = \sum_{i=1}^d  a_i \: f_i(\mathbf x).
    \]
    We suppose that there exist \(k \geq1\) such that the set \(\mathcal{B}_k\) spans of polynomials of degree \(k\). Let \(f \in \mathcal{P}^{k+1}\). We can always write it as \(f = f_k + h\), with \(f_k \in \mathcal{P}^{k} \) and where \(h\) is a homogeneous polynomial of degree \(k+1\). Let us denote \(I_k\) the set of indices required to decompose any \(f_k\) in terms of function in  \(\mathcal{B}_k\). By induction, we can write
    \[
    f_k(\mathbf x) = \sum_{i \in I_k} a_i \: p_{\alpha^{(i)}}(\mathbf x), \text{ with } p_{\alpha^{(i)}} \in \mathcal{B}_k.
    \]
    We note that \(h\) is of the form
    \[
    h(\mathbf x) = \sum_{ \substack{\alpha \in \mathbb{N}^d \\ |\alpha| = k+1} } c_\alpha \prod_{i=1}^d x_i^{\alpha_i} = \sum_{ \substack{\alpha \in \mathbb{N}^d \\ |\alpha| = k+1} } c_\alpha x_{i^{*}} \prod_{\substack{i=1\\ i \neq i^{*}}}^d x_i^{\alpha_i},
    \]
    where \( i^{*} = i^{*}(\alpha) = \min \{i : \alpha_i >0 \}\). By induction, we can write
    \[
    h(\mathbf x) = \sum_{ \substack{\alpha \in \mathbb{N}^d \\ |\alpha| = k+1} } \left( \sum_{i=1}^d \tilde c_i f_i(\mathbf x) \right) \left( \sum_{i \in I_k} \hat c_i \: p_{\alpha^{(i)}}(\mathbf x) \right), \text{ with } p_{\alpha^{(i)}} \in \mathcal{B}_k.
    \]
    By developing the product and combining the sums, we obtain
    \[
    f(\mathbf x) = \sum_{i \in I_{k+1}} a_i \: p_{\alpha^{(i)}}(\mathbf x), \text{ with } p_{\alpha^{(i)}}\in \mathcal{B}_{k+1}.
    \]
\end{proof}

\begin{lemma} \label{Spectrum of JS}
    The matrix \(JS\) is similar to the skew symmetric matrix \(S^{1/2}JS^{1/2}\). I.e.
    \[
    S^{1/2} (JS) S^{-1/2} = S^{1/2} J S^{1/2}
    \]
    In particular, its eigenvalues are purely imaginary and \(\sigma(JS) = \{ \pm \mathrm{i} \}\). Moreover, the matrix \(JS\) is diagonalizable over \(\mathbb{C}\).
\end{lemma}
\begin{proof}
    We have \(S^{1/2} = \text{diag}(\sqrt{s_1}, \dots, \sqrt{s_d})\). Since \(s_i > 0\) for all \(i\), \(S^{1/2}\) is invertible. Define the matrix \(M\) as
\[
M = S^{1/2} (JS) S^{-1/2} = S^{1/2} J S^{1/2}
\]
Since \(JS\) and \(M\) are similar, they share the same characteristic polynomial and eigenvalues. The entries of \(M\) are given by
\[
M_{ij} = \sqrt{s_i} J_{ij} \sqrt{s_j}
\]
We consider the non-zero entries of \(J\)
\begin{itemize}
    \item[] For \(1 \leq i \leq n\) and \(j = d-i+1\)
    \[ M_{i, d-i+1} = \sqrt{s_i} \left( \frac{1}{\sqrt{s_i s_{d-i+1}}} \right) \sqrt{s_{d-i+1}} = 1. \]
    
    \item[] For \(n+1 \leq i \leq d\) and \(j = d-i+1\)
    \[ M_{i, j} = \sqrt{s_{i}} \left( -a_j \right) \sqrt{s_j} = \sqrt{s_{d-j+1}} \left( -\frac{1}{\sqrt{s_j s_{d-j+1}}} \right) \sqrt{s_j} = -1. \]
\end{itemize}

The matrix \(M\) is block diagonal matrix of the form
\[
M = \begin{pmatrix} 0 & 1 \\ -1 & 0 \end{pmatrix}.
\]
Specifically, for each pair of indices \(\{i, d-i+1\}\), the subspace \(V_i = \text{span}\{e_i, e_{d-i+1}\}\) is invariant under \(M\). The restriction \(M|_{V_i}\) in the basis \((e_i, e_{d-i+1})\) is exactly block \( B = \begin{pmatrix} 0 & 1 \\ -1 & 0 \end{pmatrix}\). The eigenvalues for each block \(B\) are \(\pm i\). Since there are \(n\) such blocks, the spectrum of \(M\) is 
\[
\sigma(JS) = \{ \pm i\},
\]
where each eigenvalues has multiplicity \(n\). Finally, we note that the matrix \(M = S^{1/2} J S^{1/2}\) is real and skew-symmetric (\(M^T = -M\)). We note that every real skew-symmetric matrix is a normal matrix (\(M^T M = M M^T = -M^2\)). Since by the spectral theorem, any normal matrix is unitarily diagonalizable over \(\mathbb{C}\), \(JS\) is diagonalizable over \(\mathbb{C}\).

\end{proof}

\begin{proposition}
    For all \(k\geq1\), we have
    \[
     \inf_{u\in B^k \backslash\{0\}} R(u) \geq 1 + k\, s_d.
    \]
\end{proposition}
\begin{proof}
    
    From Proposition \ref{Spectrum of T} and Lemma \ref{Spectrum of JS}, we can see that the spectrum of the operator \(T = -\mathcal{A} = x^T SJ\nabla\) is an integer linear combination of \(\pm i\). Hence, on \(\mathcal{P}^{k}\)
    \[
        \sigma(T_{|\mathcal{P}^{k}}) = \left\{ i(k-2r) | \:k\geq 1, \: r=0,\dots,k \right\}.
    \]
    Since \(T \) is skew-symmetric in \(L^2(\mu)\), we can deduce the spectrum of \(T^{\#}T = -T^2\) from the spectrum of \(T\). Namely, on \(\mathcal{P}^{k}\), we have
    \[
    \sigma(T^{\#}T_{|\mathcal{P}^{k}}) = \left\{ (k-2r)^2| \:k\geq 1,\:  r=0,\dots,k \right\}.
    \]
    We cannot have a contribution to the eigenvector of \(T\) associated with the eigenvalue \(0\) on \(B^{k} \). Hence, since \(T^{\#}T\) is self adjoint and has a discrete spectrum, we can write any \( u\in B^{k}\) in an eigenbasis
    \[
    u = \sum u_i v_i.
    \]
    Therefore, the Rayleigh quotient becomes
    \[
    \frac{\langle Tu,Tu\rangle}{\langle u,u\rangle} = \frac{\sum \lambda_i u_i^2 }{\sum u_i^2} \geq \inf \{  \lambda \in \sigma(T^{\#}T _{|B^k}) \} .
    \]
    Finally, since all non zero eigenvalues of \(T\) on \(B^{k} \) are bigger than 1, we have
    \[
     \inf_{u\in B^{k}_0\backslash\{0\} }\frac{\langle Tu,Tu\rangle}{\langle u,u\rangle} \geq  1.
    \]
    In conclusion, we have for all \(k\geq1\)
    \[
     \inf_{u\in B^{k}_0 \backslash\{0\} } R(u) \geq 1 + k s_d.
    \]
    \end{proof}

\subsection{Spectral gap on \texorpdfstring{$\mathcal{P}^k_0$}{Pk}} \label{Spectral gap on Pk}
    We now have all the tools needed to prove the main results.
    
\begin{proof}[Proof of Theorem \ref{Main result d=2n asymptotic}]

To show the lower bound under the limit, we consider an increasing sequence \(\left( \delta_n\right)_{n\geq 1}\) such that \(\delta_n \to \infty\) and the eigenfunction \(f_{n}\) minimizing the Rayleigh quotient for a given \(\delta_n\). We have from Section \ref{Spectral gap on A2k} that for all \( \delta_n >0\),
    \[
    \inf_{\substack{ f \in L^2_0(\pi) \\ \langle f,f\rangle = 1 }} R_{ \delta_n}(f) \leq \gamma_{\min}.
    \]
    Moreover, \(\mathcal{R} \succeq \mathcal{L}\), hence for \(k \geq \gamma_{\min} s_d^{-1}\), we have
        \[
    \inf_{\substack{ f \in \mathcal{H}_{k} \\ \langle f,f\rangle = 1 }} R_{ \delta_n}(f) \geq k s_d \geq \gamma_{\min}.
    \]
    Therefore, for all \( \delta_n >0\), \(f_{n}\) must lie in the finite dimensional space \(\bigoplus_{k=1}^{ \lfloor \gamma_{\min} s_d^{-1} \rfloor } \mathcal{H}_k\).
     
    We now decompose \(f_n\) in a direct sum to use the results from Section \ref{Spectral gap on A2k}. Since, \(L^2_0(\pi) = \ker(\mathcal{A}) \oplus \overline {\text{Image}(\mathcal{A})}\), we can decompose any observable \(f\in L^2_0(\pi)\) as
    \[
    f = g + h,\quad g \in \ker(\mathcal{A}) , \:h\in \overline {\text{Image}(\mathcal{A})}.
    \]
    Without loss of generality, we can suppose that \(\langle f,f\rangle = 1\). Since, \(g \perp h\), we must have \[\langle g,g\rangle + \langle h,h\rangle = 1. \]
     Using reversibility and the fact that \(\mathcal{A} g = 0\), the Rayleigh quotient writes as
    \begin{equation}
        R(f) = \langle g, \nabla^{\#}\nabla g\rangle  + 2\langle g, \nabla^{\#}\nabla h \rangle + \langle h, (\nabla^{\#}\nabla + \mathcal{A}^{\#}\mathcal{A})h \rangle. \label{eq Rayleigh quotient decomposition into two blocks}
    \end{equation}

    We start by estimating the terms in the variables \(\delta_n \), we have
    \[
    R_{ \delta_n}(f) \geq \langle g, \nabla^{\#}\nabla g\rangle  - 2|\langle g, \nabla^{\#}\nabla h \rangle| + \langle h, (\nabla^{\#}\nabla +  \delta_n^2\mathcal{A}^{\#}\mathcal{A})h \rangle,
    \]
    In particular, the terms \(\langle h, (\nabla^{\#}\nabla +\delta_n^2 \mathcal{A}^{\#}\mathcal{A})h \rangle\) depends quadratically on \( \delta_n \), whereas the other terms are constant in \( \delta_n \).
    We have
    \[
      \inf_{\substack{ f \in L^2_0(\pi) \\ \langle f,f\rangle = 1 }} R_{ \delta_n}(f)= R_{ \delta_n}(f_n)  \geq \langle g_n, \nabla^{\#}\nabla g_n\rangle  - 2|\langle g_n, \nabla^{\#}\nabla h_n \rangle| + \langle h_n, (\nabla^{\#}\nabla + \delta_n^2 \mathcal{A}^{\#}\mathcal{A})h_n \rangle.
    \]
    Since the left hand-side explodes when \(h_n\) does not vanish in the \(L^2(\pi)\)-limit and \( f_n  \) lies in a finite dimensional space for all \(n\), we must have \(f_n \to \lim_n g_n = g \in \ker(\mathcal{A})\). In particular, we have \(2|\langle g_n, \nabla^{\#}\nabla h_n \rangle| \to 0\). Hence,
    \[
    \lim_{\delta \to \infty} \inf_{\substack{ f \in L^2_0(\pi) \\ \langle f,f\rangle = 1 }} R_{\delta_n}(f) = \lim_{n\to \infty} R_{\delta_n}(f_n) \geq    \lim_{n\to \infty} R_{\delta_n}(g_n) \geq  \lim_{n\to \infty}\inf_{\substack{ g \in \ker(\mathcal{A}) \\ \langle g,g\rangle = 1, \:\langle g,1\rangle = 0 }} R_{\delta_n}(g) \geq \gamma_{\min},
    \]
    where in the last inequality, we use the Proposition \ref{sp gap on ker(A)}. Finally, by the Squeeze theorem, we have
        \[
    \lim_{\delta \to \infty} \lambda_\delta = \gamma_{\min}.
    \]   
\end{proof}

Note that we can prove a non-asymptotic lower bound for the Rayleigh quotient on \(k\)-th Wiener Chaos using the analysis done in Sections \ref{section reduction to Wiener Chaos}, \ref{Spectral gap on A2k} and \ref{Spectral gap on Bk} except for finitely many \(k\). We detail this fact in the following proposition.

\begin{proposition} \label{prop cv non asymptotic pairwise matching}
    Let \(S \in \mathcal{S}_d^{>0}(\mathbb{R}), J \in \mathcal{A}_d(\mathbb{R})\) be as in Section \ref{Pairwise matching in a general dimension}. For all \(\delta \ge 1\), we have
    \[
    \inf_{f \in L^2_0(\pi)}R_\delta(f) \geq \min\{1, \min_{k\in 2\mathbb{N}\cap[2,s_d^{-1}]} \inf_{f \in \mathcal{H}_k} R_\delta(f_k)  \}.
    \]
\end{proposition}
\begin{proof}
    From Section \ref{section reduction to Wiener Chaos}, we already know that
    \[
    \inf_{f \in L^2_0(\pi)}R_\delta(f) \geq  \inf_{k\ge 1} \inf_{f \in \mathcal{H}_k} R_\delta(f_k).
    \]
    Moreover, from Section \ref{Spectral gap on A2k} and \ref{Spectral gap on Bk}, we have for all \(k \ge 1\)
    \[
    \inf_{f \in \mathcal{H}_{2k+1}} R_\delta(f_k) \geq \delta^2 +(2k+1)s_d \ge 1.
    \]
    Additionally, since \(-\mathcal{R}_\delta \succeq -\mathcal{L}\) and \(-\mathcal{L}_{|\mathcal{H}_k} \succeq k s_d \), we have for all \(k\ge s_d^{-1}\)
  \[
    \inf_{f \in \mathcal{H}_{k}} R_\delta(f_k) \geq \inf_{f \in \mathcal{H}_{k}} \frac{\inner{\nabla f}{\nabla f}}{\inner{f}{f}} \geq k \,s_d \geq 1.
    \]
    Combining all this lower bound leads to the reduction from the Proposition's statement.
\end{proof}
This result implies in particular that odd moments converges to equilibrium at rate at least \(1 +s_d\). A natural way to derive a general non asymptotic lower bound is to decompose \(\mathcal{H}_k\) into a direct sum of \(\ker(\mathcal{A}) \cap \mathcal{H}_k\) and \(B^k = (A^k)^\perp = \{ p \in \mathcal{H}_k : p \perp a, \forall a \in A^k \}\) and use the analysis from Sections \ref{Spectral gap on A2k} and \ref{Spectral gap on Bk}. One obstruction we met when \(\delta = 1\) is that the cross terms \(2\langle g, \nabla^{\#}\nabla h \rangle \) in \eqref{eq Rayleigh quotient decomposition into two blocks} is difficult to control and lead to a non positive of lower bound when treated non carefully.

For instance, when \(\delta = 1\), controlling the cross terms with Cauchy-Schwarz and Young's inequality gives a lower bound of the Rayleigh quotient on the second Wiener Chaos of \(\gamma_{\min}/9\) but a negative bound for large \(k\)-th chaos. From various arguments, it seems that the true spectral gap should lie in the second Wiener Chaos, hence we believe that choosing \(\delta =1\) leads to significant improvement as confirmed in the numerical experiments.

On the other hand, we can still obtain the non asymptotic lower bound presented in Theorem \ref{Main result d=2n non asymptotic} by taking \(\delta = O(s_d^{-1})\). We prove the key estimate in Proposition \ref{prop cv non asymptotic remainding cases}.

\begin{proposition} \label{prop cv non asymptotic remainding cases}
    Let \(k \in \mathbb{N} \cap[1, (2s_d)^{-1}]\). Let \(\delta \) denotes the strength of the Stratonovich perturbation. If 
    \[
    \delta^2 \ge k\frac{ (\gamma_{\min})^2 + 16 s_1^2 - 4 \gamma_{\min} s_d }{2\gamma_{\min}},
    \]
    Then 
    \[
    \inf_{f\in \mathcal{H}_{2k}} R_\delta(f) \geq k\frac{\gamma_{\min}}{2}.
    \]
\end{proposition}
\begin{proof}
    For any \(\delta >0\), we consider the Stratonovich perturbation associated with \(J_{\delta} = \delta J\) as in Section \ref{Pairwise matching in a general dimension}. By Lemma \ref{Wiener Chaos preservation}, \(\mathcal{R}_\delta(\mathcal{H}_{2k}) \subset \mathcal{H}_{2k}\). Since \(\mathcal{H}_{2k}\) is a finite dimensional space and using the decomposition from Lemma \ref{H_k = A^k oplus B^k}, there exists \(f_\delta = g_\delta+h_\delta \in \mathcal{H}_{2k} = A^{2k} \oplus B^{2k}\) such that \(\norm{f_\delta}=1\) and
    \[
     \inf_{f\in \mathcal{H}_{2k}} R_\delta(f)  = R(f_\delta).
    \]
    We now derive an explicit lower bound of the Rayleigh quotient \(R(f_\delta)\). We use the orthogonality of \(g\perp h\) in the first inequality. For the second one, we use the analysis in Section \ref{Spectral gap on A2k} and \ref{Spectral gap on Bk} and apply the Cauchy-Schwarz inequality on the cross term. We have
    \[
    \begin{split}
           R(f_\delta) &\ge \langle g_\delta, \nabla^{\#}\nabla g_\delta\rangle  + \langle h_\delta, (\nabla^{\#}\nabla + \delta^2 \mathcal{A}^{\#}\mathcal{A})h_\delta \rangle - 2|\langle g_\delta, \nabla^{\#}\nabla h_\delta \rangle|, \\
           & \ge k \gamma_{\min} \norm{g_\delta}^2 + (2k s_d + \delta^2) \norm{h_\delta}^2 - 4ks_1 \norm{g_\delta} \norm{h_\delta},\\
           & \ge \min_{\substack{a^2 +b^2 = 1\\ a,b \ge 0}} k \gamma_{\min} a^2 + (2k s_d + \delta^2) b^2 - 4ks_1 ab,\\
           &= \frac{k \gamma_{\min}+ 2 k s_d + \delta^2 - \sqrt{(k \gamma_{\min}- 2 k s_d - \delta^2)^2 + 16 k^2 s_1^2 }}{2}.
    \end{split}
    \]
    The last inequality is a classical minimization problem and is detailed in Lemma \ref{constraint pb lemma}. A straightforward calculation detailed in Lemma \ref{F(x,a,b) >= a/2}, with \(a = k\gamma_{\min}, b = 2ks_d, c= 4ks_1\) and \(x = \delta^2\), shows that if 
    \[
    \delta^2 \ge k\frac{ (\gamma_{\min})^2 + 16s_1^2 - 4 \gamma_{\min} s_d }{2\gamma_{\min}},
    \]
    we have
    \[
    \frac{k \gamma_{\min}+ 2 k s_d + \delta^2 - \sqrt{(k \gamma_{\min}- 2 k s_d - \delta^2)^2 + 16 k^2 s_1^2 }}{2} \ge k \frac{\gamma_{\min}}{2}.
    \]
    We can therefore conclude
    \[
     \inf_{f\in \mathcal{H}_{2k}} R_\delta(f)  \ge k \frac{\gamma_{\min}}{2}.
    \]
    
\end{proof}

\begin{proof}[Proof of Theorem \ref{Main result d=2n non asymptotic}]
    From Proposition \ref{prop cv non asymptotic pairwise matching}
    \[
    \inf_{f \in L^2_0(\pi)}R(f) \geq \min\{1, \min_{k \in \mathbb{N} \cap[1, (2s_d)^{-1}]} \inf_{f \in \mathcal{H}_{2k}} R(f)  \}. 
    \]
    We have for all \(k \in \mathbb{N} \cap[1, (2s_d)^{-1}] \)
    \[
    \delta^2 \ge \frac{ (\gamma_{\min})^2 + 16s_1^2 - 4 \gamma_{\min} s_d }{4\gamma_{\min}s_d} \ge k\frac{ (\gamma_{\min})^2 + 16 s_1^2 - 4 \gamma_{\min} s_d }{2\gamma_{\min}}. 
    \]
    Hence by Proposition \ref{prop cv non asymptotic remainding cases}, we have \(\inf_{f\in \mathcal{H}_{2k}} R_\delta(f) \geq k\frac{\gamma_{\min}}{2}\). Therefore,
    \[
    \inf_{f \in L^2_0(\pi)}R(f)  \geq \min \{1, \frac{\gamma_{\min}}{2} \} \ge \frac{\gamma_{\min}}{2},
    \]
    where in the last inequality we use \(\gamma_{\min} \leq 2 s_1 \leq 2\).
\end{proof}

\begin{lemma}\label{H_k = A^k oplus B^k}
    Let \(k\ge 1\) and consider \(A^k = \ker(\mathcal{A}) \cap \mathcal{H}_k\) and \(B^k = (A^k)^\perp = \{ p \in \mathcal{H}_k : p \perp a, \forall a \in A^k \}\). We have
    \[
    H_k = A^k \oplus B^k.
    \]
\end{lemma}
\begin{proof}
    By Lemma \ref{Wiener Chaos preservation}, \(\mathcal{A}(\mathcal{H}_k) \subset \mathcal{H}_k\). Because \(\mathcal{H}_k \) is finite dimensional, we can use the rank nullity theorem to show that
    \[
    \dim(\mathcal{H}_k) = \dim\left( \ker(\mathcal{A}_{|\mathcal{H}_k}) \right) + \dim\left( \text{Im}(\mathcal{A}_{|\mathcal{H}_k}) \right) = \dim\left(\ker(\mathcal{A}_{|\mathcal{H}_k}) \oplus \text{Im}(\mathcal{A}\cap  \text{Im}(\mathcal{A}_{|\mathcal{H}_k}) \right) .
    \]
    Note that the right-hand-side is a finite dimensional subspace of the finite dimensional vector space \(\mathcal{H}_k\). Moreover, this subspace is of the same dimension as the whole space, hence, they are equal. We conclude that
    \[
    \mathcal{H}_k = A^k \oplus B^k.
    \]
\end{proof}

\begin{lemma}\label{constraint pb lemma}
    For all weights \(c_1,c_2,c_3 >0\), we have
    \[
    \inf_{\substack{a,b \geq 0 \\ a^2 + b^2 = 1} }c_1 a^2  + c_2b^2 - c_3 ab = \frac{c_1 + c_2 - \sqrt{(c_1 - c_2)^2 + c_3^2}}{2}.
    \]
\end{lemma}
\begin{proof}
    We first note that the set \(\{ a,b \geq 0 : a^2+b^2 =1 \}\) is compact, so the infimum is a minimum. We can rewrite the function to be minimized as a quadratic form
\[
f(a,b) = 
\begin{pmatrix} a & b \end{pmatrix} 
\begin{pmatrix} c_1 & -c_3/2 \\ -c_3/2 & c_2 \end{pmatrix} 
\begin{pmatrix} a \\ b \end{pmatrix}.
\]

Minimizing this over the unit circle \(a^2 + b^2 = 1\) is equivalent to finding the smallest eigenvalue \(\lambda_{\min}\) of the symmetric matrix
\[
H = \begin{pmatrix} c_1 & -c_3/2 \\ -c_3/2 & c_2 \end{pmatrix}
\]

Its eigenvalues are
\[
\lambda_{\pm}  = \frac{c_1 + c_2 \pm \sqrt{(c_1 - c_2)^2 + c_3^2}}{2}.
\]
Clearly, \(\lambda_{\min} = \lambda_-\). To ensure this minimum is valid for the constraint \(a, b \geq 0\), we verify that its corresponding eigenvector lies in the first quadrant.
The corresponding eigenvector \(v_-\) is
\[
v_- = \begin{pmatrix}
    c_3 \\
    2(c_1 - \lambda_-)
\end{pmatrix}.
\]
Since \(c_3 > 0\) by hypothesis and 
\[
c_1 - \lambda_- = \frac{c_1 - c_2 + \sqrt{(c_1 - c_2)^2 + c_3^2}}{2} \geq 0,
\]
all component of the eigenvector are positive, which gives that our minimizing \(a,b\) are indeed positive.
\end{proof}

\begin{lemma}\label{F(x,a,b) >= a/2}
    Let \(a,b,x> 0\) and consider 
    \[
    F(x,a,b) = \frac{a+b+x - \sqrt{(a-b-x)^2 + c^2}}{2}.
    \]
    If
    \[
    x \ge \frac{a^2 + c^2 - 2ab}{2a},
    \]
    Then \(F(x,a,b) \ge a/2\).
\end{lemma}
\begin{proof}
    We have
    \[
    F(x,a,b) \ge a/2 \iff (b+x) \ge  \sqrt{(a-b-x)^2 + c^2} \iff -a^2 + 2 a b - c^2 + 2 a x \ge 0.
    \]
    And we obtain
    \[
    F(x,a,b) \ge a/2 \iff x \ge \frac{a^2 + c^2 -2ab}{2a}.
    \]
\end{proof}

\section{Numerical experiments} \label{Numerical experiment}

\subsection{Two dimensional experiments}

In this section, we first illustrate the theoretical result presented in \ref{2D Results}. In particular, we show that we obtain a significant drop in the convergence rate with a suboptimal choice \(\delta = \lambda^{-1/4}\) instead of \( \delta =\lambda^{-1/2} \).

Then, we provide an example where we aim to recover the density of the invariant Gibbs measure in a non Gaussian case.

For the first experiment, we use the Euler–Maruyama scheme to discretize the dynamics \eqref{Reversible SDE 2d}
\[
     \mathbf{X}_{n+1} = \mathbf{X}_{n} - B \mathbf{X}_{n}\, h+ \delta JS \mathbf{X}_{n}\,\sqrt{2h} \beta_n+ \sqrt{2h} \mathbf{W}_{n}, \quad \beta_n\overset{i.i.d.}{\sim}\mathcal{N}(0,1),\: \mathbf{W}_{n} \overset{i.i.d.}{\sim}\mathcal{N}(0,I).
\]
In Figure \ref{Critical vs subcritical 2d}, we use a Monte-Carlo estimation with \(M = 10^4\) trajectories to approximate the expectation of the second moment for the different dynamics: the original process (OP), the nonreversible process (NRP) and the Stratonovich reversible process (RP). We also plot a reference slope corresponding to \(O(e^{-\gamma t})\)). To simulate the dynamics, we use an Euler-Maruyama scheme with \(h = 10^{-4}\). We compare the speed of convergence to equilibrium for different original spectral gap \(\lambda\) and different perturbation scaling \(\delta\). Namely, we choose \( \delta =\lambda^{-1/2} \) in the upper image and \(\delta = \lambda^{-1/4}\) in the lower one and we compute 
\begin{equation}
    e(t) = \Big\vert  \frac{1}{M}\sum_{i=1}^M f_2(\textbf{X}_t^{(i)}) - \pi(f_2) \Big\vert, \quad f_2(\textbf{x}) = \sum_{i=1}^d \textbf{x}_i^2. \label{error e_2(t)}
\end{equation}

We notice that the suboptimal scaling is insufficient since the convergence rate of the Stratonovich perturbation in this case collapse to the original one for small \(\lambda\).

\begin{figure}[H]
    \centering
    
    \begin{minipage}{0.48\textwidth}
        \centering
        \captionof*{figure}{Optimal choice with \(\lambda = 10^{-1}\)}
         \vspace{-0.2cm}
        \includegraphics[width=0.8\linewidth]{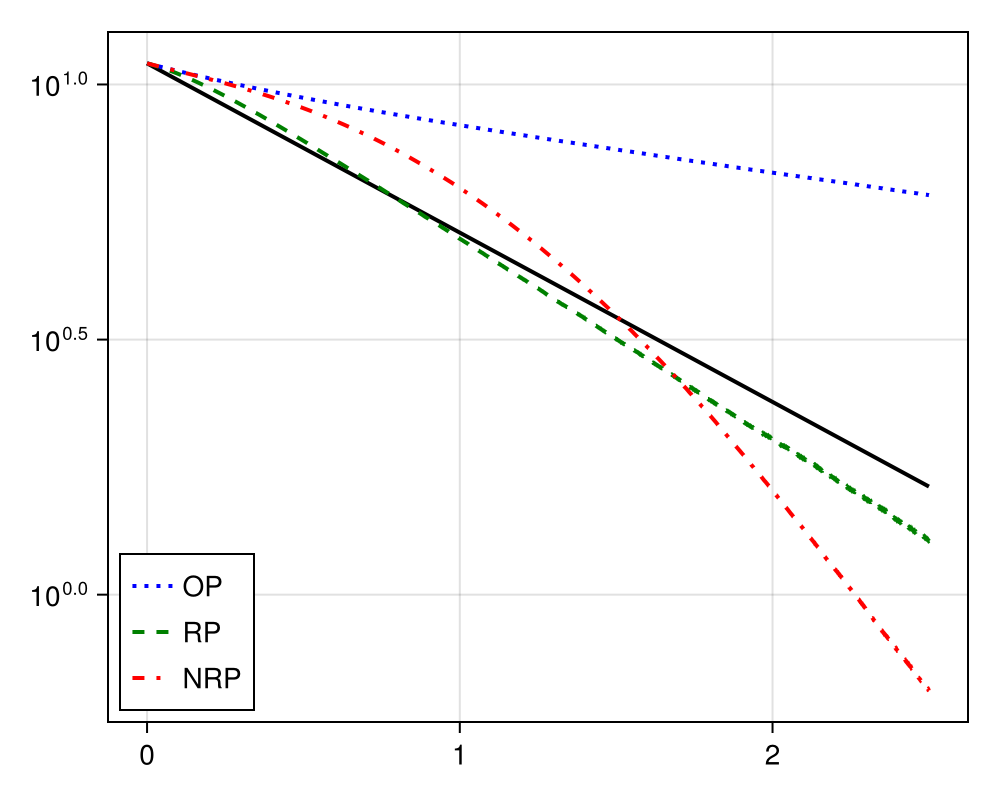}
    \end{minipage}
    \hfill
    \begin{minipage}{0.48\textwidth}
        \centering
        \captionof*{figure}{Optimal choice with \(\lambda = 10^{-3}\)}
         \vspace{-0.2cm}
        \includegraphics[width=0.8\linewidth]{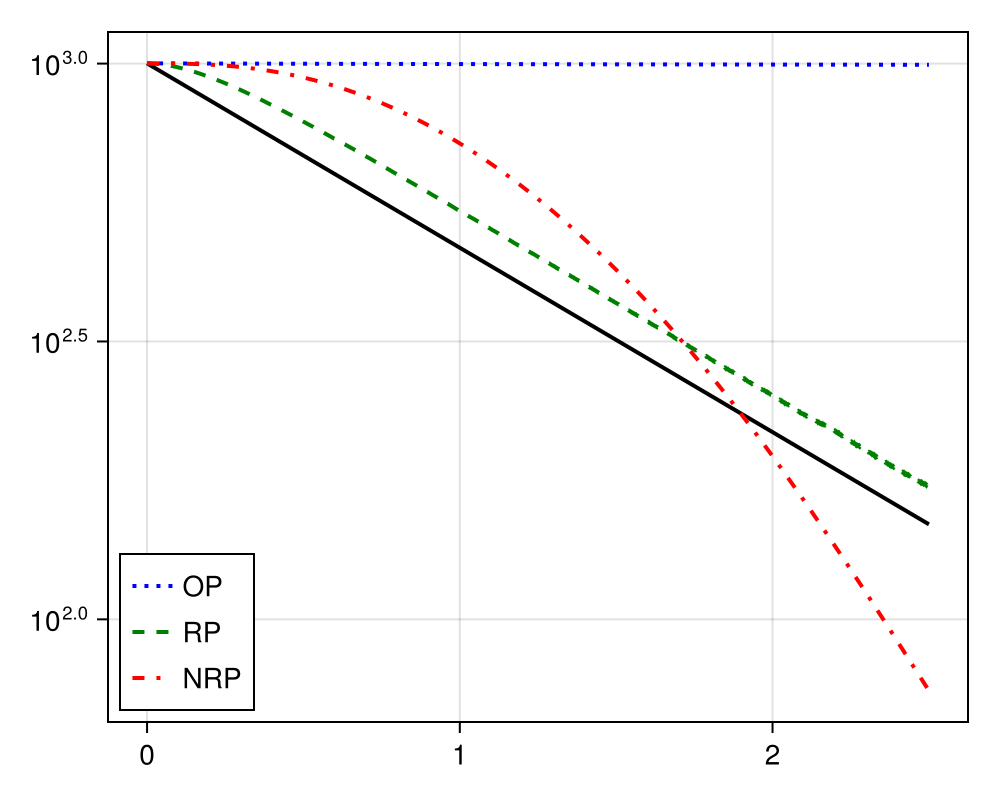}
    \end{minipage}

    \vspace{0.25cm}

    \begin{minipage}{0.48\textwidth}
        \centering
        \captionof*{figure}{Suboptimal choice with \(\lambda = 10^{-1}\)}
         \vspace{-0.2cm}
        \includegraphics[width=0.8\linewidth]{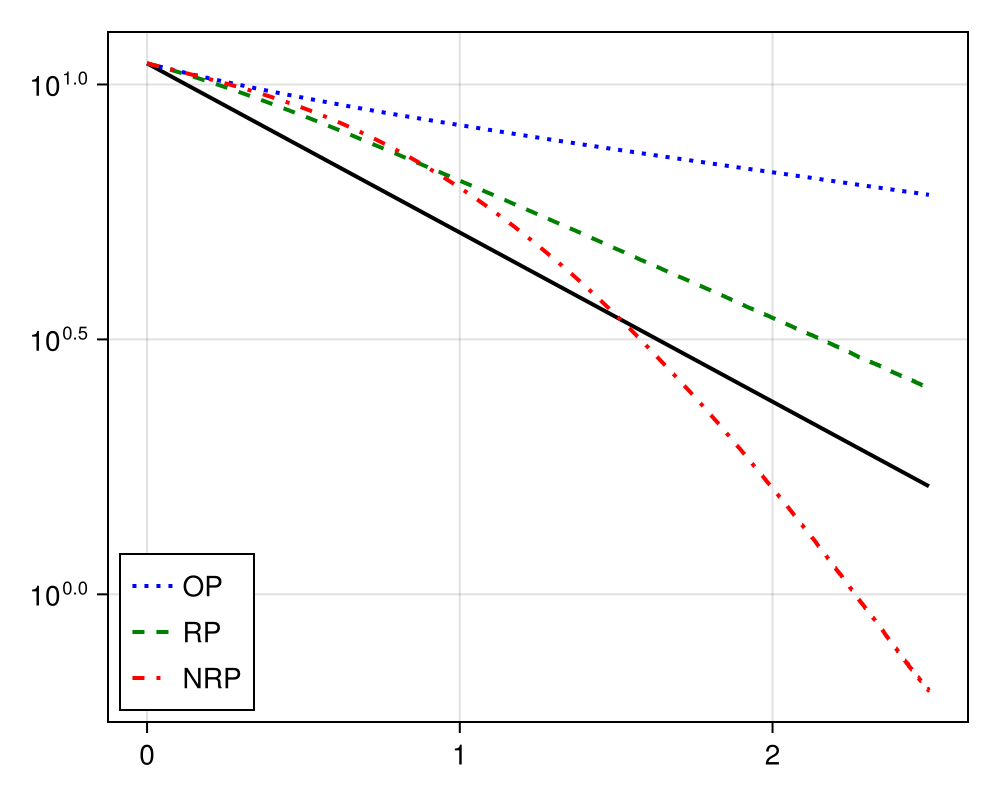}
    \end{minipage}
    \hfill
    \begin{minipage}{0.48\textwidth}
        \centering
         \captionof*{figure}{Suboptimal choice with \(\lambda = 10^{-3}\)}
         \vspace{-0.2cm}
        \includegraphics[width=0.8\linewidth]{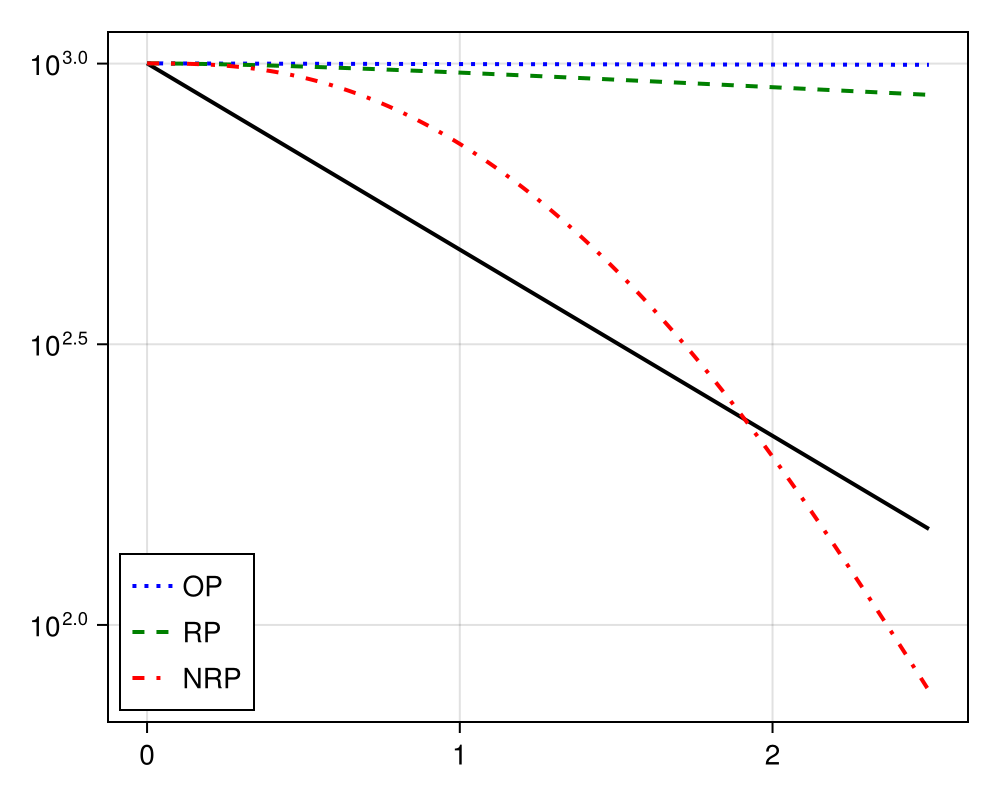}
    \end{minipage}

    \caption{Comparison between an optimal and a suboptimal scaling.}
 \label{Critical vs subcritical 2d}
\end{figure}

In the next experiment, we provide an example where the Stratonovich perturbation helps overcoming metastability issues. We consider a Gaussian mixture with two modes for the invariant measure. We define via its density
\[
\rho(\mathbf x) \propto \sum_{i=1}^{2}
\exp\left(- (\mathbf x - \mathbf{x}_i)^\top S (\mathbf x - \mathbf{x}_i)\right),
\]
where
\[
\mathbf x_1 = \begin{pmatrix}
    1\\0
\end{pmatrix}, \mathbf x_2 = -\begin{pmatrix}
    1\\0
\end{pmatrix},\: S = \begin{pmatrix}
    10&0\\ 0&1
\end{pmatrix}.
\]
In this case, the potential is defined as the \(\log\) of the density of the invariant measure
\[
V(\mathbf x) = -\log \left(
\rho(\mathbf x)
\right).
\]
We apply the Stratonovich perturbation and the nonreversible perturbation with scaling \(\delta = 4\) and use a histogram with \(100\) bins to recover the invariant measure after a time \(T = 50\) with size step \(h=10^{-4}\). We compare the effectiveness of the exploration for the different processes in figure \ref{visu density GM two nodes}. In this example, it appears that using the reversible perturbation helps to overcome energetic barrier.

\begin{figure}[H]
    \centering

    \begin{minipage}{0.48\textwidth}
        \centering
        \captionof*{figure}{ \small a) True density}
         \vspace{-0.1cm}
        \includegraphics[width=0.6\linewidth]{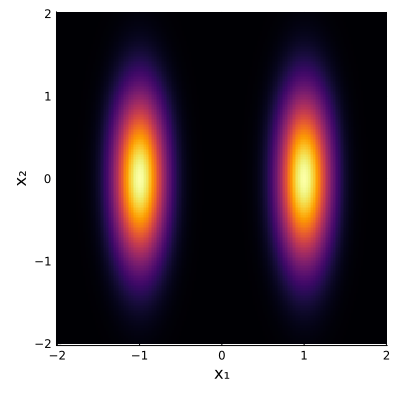}
    \end{minipage}
    \hfill
    \begin{minipage}{0.48\textwidth}
        \centering
        \captionof*{figure}{ \small b) OP \eqref{OLD}}
         \vspace{-0.1cm}
        \includegraphics[width=0.6\linewidth]{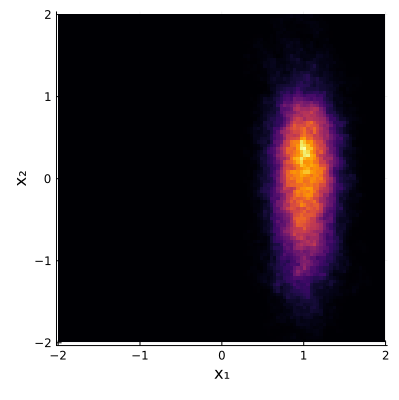}
        
    \end{minipage}

    \vspace{0.25cm}

    \begin{minipage}{0.48\textwidth}
        \centering
        \captionof*{figure}{ \small c) NRP \eqref{nonreversible SDE}}
         \vspace{-0.1cm}
        \includegraphics[width=0.6\linewidth]{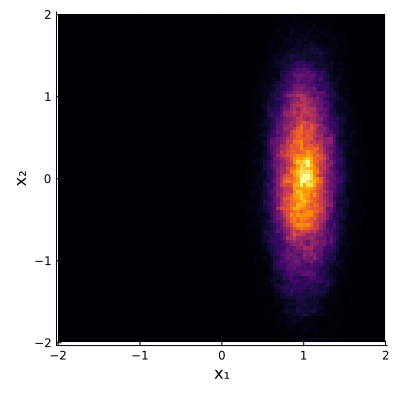}
 
    \end{minipage}
    \hfill
    \begin{minipage}{0.48\textwidth}
        \centering
        \captionof*{figure}{ \small d) RP \eqref{Stratonovich SDE}}
         \vspace{-0.1cm}
        \includegraphics[width=0.6\linewidth]{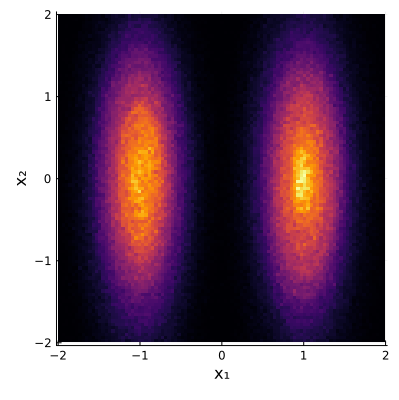}

    \end{minipage}

    \captionof{figure}{Approximation of the Gaussian mixture constructed by tracking different processes.}
 \label{visu density GM two nodes}
\end{figure}

\subsection{Multidimensional Gaussian experiments}\label{Multidimensional Gaussian experiment}

In the first experiment, we compare the skew symmetric matrix built using Algorithm \ref{Aglo pairwise} to a random choice of skew symmetric matrix \(J\) when the drift matrix is the discrete Laplacian
\begin{equation}
    S_d = \begin{pmatrix}
    2 & -1 & \\ -1 & 2 & -1 & \\ & \ddots & \ddots & \ddots\\ && \ddots & \ddots & -1 \\ &&&-1 & 2
\end{pmatrix} \in \mathcal{S}_d^{>0}(\mathbb R). \label{discrete Laplacian}
\end{equation}

The eigenvalues of this matrix are
\[
\lambda_k = 2 \left(1 - \cos \frac{k \pi}{d+1} \right), \quad k = 1, 2, \dots, d.
\]
We can see that the smallest eigenvalue of \(S_d\) is \(\lambda_1 = 2( 1-\cos(\frac{ \pi}{d+1})) \approx \frac{\pi^2}{(d+1)^2}\).

We build the random matrix by taking the skew-symmetric part of a standard Gaussian matrix, with the seed \(42\) in Julia. We also compare the Stratonovich perturbation to the nonreversible one. The latter being built using the algorithm presented in \cite{NRPLelievre2013}.

The result is presented in Figure \ref{random choice}. The black line is reference of the optimal convergence rate according to \cite{NRPLelievre2013}. We see that the random perturbation suffers from instability and does not improve significantly the convergence.

\begin{figure}[H]
\begin{minipage}{0.48\textwidth}
        \centering
    \includegraphics[width=1\linewidth]{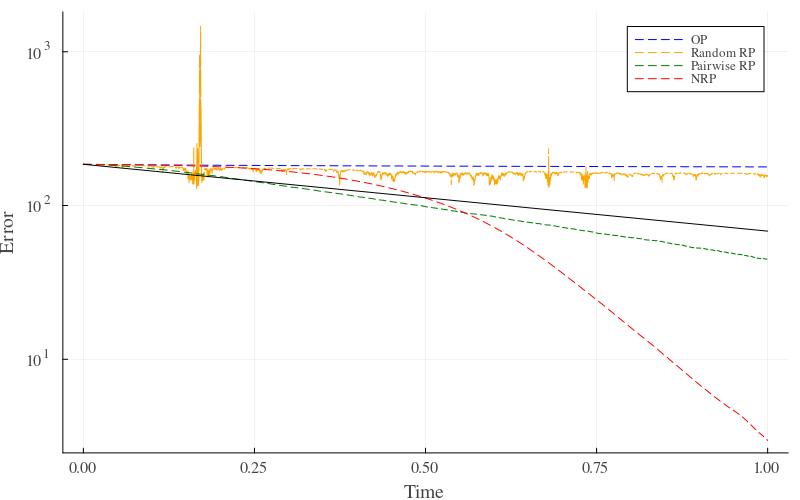}
    \captionof{figure}{Error \eqref{error e_2(t)} as function of time for the discrete Laplacian in dimension \(d=100\) over \(M = 10^4\) trajectories simulated using EM with a time step \(h = 10^{-5}\). The initial condition is \(\mathbf{X}_0 = \mathbf{0}\)}
    \label{random choice}
\end{minipage}
\hfill
\begin{minipage}{0.48\textwidth}
        \centering
    \includegraphics[width=1\linewidth]{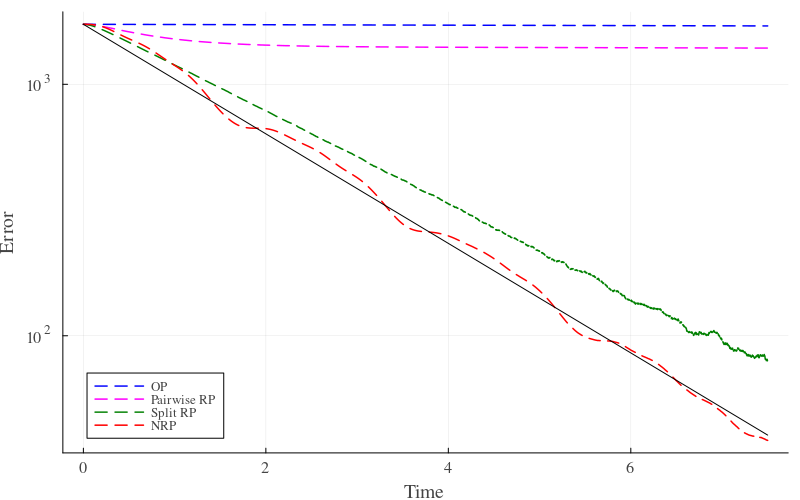}
    \captionof{figure}{Error \eqref{error e_2(t)} as function of time for the pathological example in dimension $d=4$ with a Monte-Carlo scheme ($M =10^5$) and a Euler–Maruyama method ($h = 10^{-4}$). The initial condition is \(\mathbf{X}_0 = \mathbf{0}\).}
    \label{pathological example}
\end{minipage}
\end{figure}

In Figure \ref{pathological example}, we illustrate that the perturbation built as in Section \ref{Pairwise matching in a general dimension} may be ineffective in some pathological case, but the construction presented in \ref{Splitting eigenvalues in a general dimension} maintains a satisfactory convergence rate. We choose a \( 4\) dimensional example given by
\[ S = \begin{pmatrix}
    1 \\ & \lambda \\ & &\lambda \\ &&& \lambda 
\end{pmatrix}, \quad \lambda = 10^{-2}.\]
Again, the black slope is a visual reference for the optimal convergence rate.

\subsection{Multidimensional non-Gaussian experiments} \label{Multidimensional non-Gaussian experiments}

We consider the stochastic differential equation in dimension one
\begin{equation}
    dX_s = -\nabla U(X_s)\,ds + \sqrt{2\beta^{-1}}\,dW_s,\label{diffusion bridge SDE}
\end{equation}
where $\beta > 0$, and the potential $U: \mathbb{R} \to \mathbb{R}$ satisfies adequate growth and smoothness conditions. We condition \((X_s)_{s\ge0}\) on the events $X(0) = x^- = 0$ and $X(s^+) = x^+ =0$. The law of the conditioned process is a probability measure $\pi$ on the infinite dimensional path space $L^2([0, s^+], \mathbb{R})$. Such process is known as a \emph{diffusion bridge}. We plot an example of a trajectory in Figure \ref{diffusion bridge plot}.

\begin{figure}[H]
    \centering
    \includegraphics[width=0.49\linewidth]{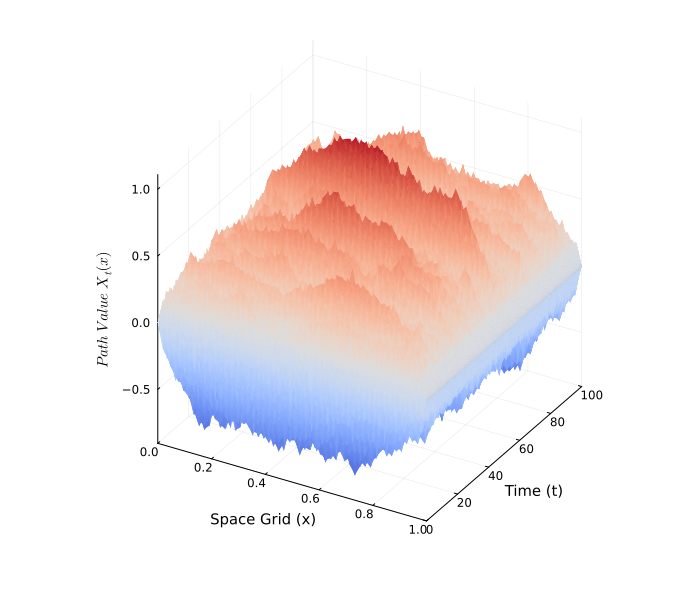}
    \includegraphics[width=0.49\linewidth]{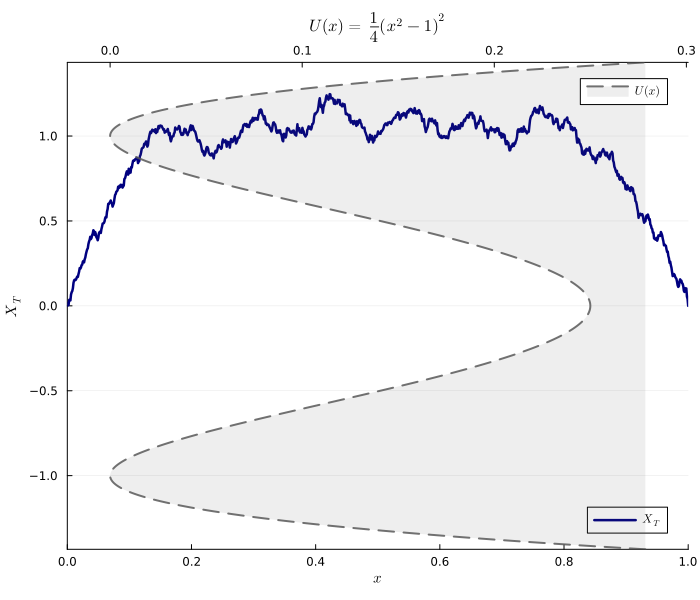}
    \caption{Left: Trajectory of the process \eqref{diffusion bridge SDE} \((X_t)_{t=0}^{100}.\)} Right: A sample \(X_{1}\) with the potential \(U\) in background.
    \label{diffusion bridge plot}
\end{figure}
When \(U= 0\), the corresponding distribution \(\pi_0\) is a centered Gaussian with precision \( S = -\frac \beta 2 \Delta\). Given a general \(U\), the measure $\pi$ is absolutely continuous with respect to $\pi_0$, with Radon–Nikodym derivative
\begin{equation}
    \frac{d\pi}{d\pi_0} \propto \exp(-\Psi), \text{ where }
    \label{eq:rn_derivative}
\end{equation}
\begin{equation}
    \Psi(x) = \frac{\beta}{2} \int_0^{s^+} G(x(s), \beta)\,ds, \quad G(x, \beta) =  \frac{1}{2}|\nabla U(x)|^2 - \frac{1}{\beta} \Delta U(x).
\end{equation}
After discretization in space, the effective potential \(V: \mathbb{R}^d \to \mathbb{R}\) can be written as
\[
V(\mathbf{x}) = \frac{\mathbf x^T S \mathbf x}{2} + \psi(\mathbf{x}),
\]
where \(\psi\) is a non linear function and \(S\) is proportional to the discrete Laplacian. See \cite[Section~5.2]{Duncan2017} for further details. In order to construct appropriate perturbation, we apply the algorithm presented in \cite{NRPLelievre2013} and Algorithm \ref{Aglo pairwise} on the quadratic part of \(V\), following the procedure described in Section \ref{Pairwise matching in a general dimension}. In the following, we consider
\[
\psi(\mathbf x)
=
\frac{\beta}{4(d+1)}
\sum_{i=1}^{d}
\sqrt{1+\left(\mathbf x_i^2-1\right)^2}, \quad \beta = 1.
\]

We compute the MSE of the temporal average for each process with the observable \(f(\mathbf{x}) = \mathbf{x}^T \mathbf{x}\). We compute the reference using a MALA scheme with for a total time of \(T= 10^5\) with stepsize \(h=10^{-5}\).

In Figure \ref{fixed budget}, we compare the efficiency of the various dynamics quantified by the MSE given a fixed computational budget \(N\), across different time-step sizes. We also compare the performance with a cheap preconditioner \(M(\mathbf{x}) = S^{-1}\), which is motivated by the fact that it corresponds to the optimal preconditioner when \(\psi=0\). For each sampler, the computational cost is defined by the number of force evaluations per time step, which constitutes the dominant computational overhead in high-dimensional spaces. This experiment relies on the numerical method employed, so we first explain our approach. 

In order to simulate the non linear dynamics above, we consider a splitting scheme for the numerical integration of the stochastic dynamics. Namely, we update the nonreversible dynamics using

\begin{equation}
        X_{n+1} = \Phi_h(X_n), \quad
\Phi_h \coloneqq  \Phi_{\Psi}^{h/2}  \circ \Phi_{P}^{h/2} \circ \Phi_{OU}^h  \circ \Phi_{P}^{h/2} \circ \Phi_{\Psi}^{h/2}. 
\end{equation}

Using the fact that Stratonovich noise preserves the chain rule, we define the following splitting for the Stratonovich perturbation
\begin{equation}
    X_{n+1} = \Phi_h(X_n), \quad
\Phi_h \coloneqq  \Phi_{\Psi}^{h/2}  \circ \Phi_{P}^{\beta_h/2} \circ \Phi_{OU}^h  \circ \Phi_{P}^{\beta_h/2} \circ \Phi_{\Psi}^{h/2}. \label{splitting stratonovich}
\end{equation}

 Here, we consider the splitting maps \(\Phi_{OU}, \Phi_P, \Phi_{\Psi} \), where each map is the integrator for the following problem respectively
\begin{align*}
  & dX_t = -S X_t dt + \sqrt{2}dW_t, \quad \text{(OU)} \\  \\
 & dX_t = -J \nabla V\left( X_t\right) dt,  \quad \text{(Perturbation)}  \\  \\
& dX_t = -\nabla \tilde \Psi\left( X_t\right) dt. \quad \text{(Non-linearity)}    
\end{align*}
We suppose that we have a fixed budget of computation \(N = 16\cdot 10^5\). We consider that the main cost of the method is to evaluate the forces. We want to compare the ratio cost/accuracy for the different dynamics.

We choose to update the non linear part of the drift using a RK\(4\) scheme, then the OU part of the dynamics using the exact flow and finally the perturbation using again a RK\(4\) scheme. 

When counting the evaluations of the forces at each step, we find that the reversibly perturbed process (RP) and the non-reversibly perturbed process (NRP) have a cost of \(16\) evaluations by time step whereas the preconditioned dynamics with a constant matrix \(S^{-1}\) (Prec) as well as the original process (OP) have a cost of \(8\) evaluations by time step. Hence, we run the former dynamics for a number of steps \(N/16\) and the latter for a number of step \(N/8\). We compute the relative error as the MSE divided by the dimension and plot the results for different time step using \(M = 10^4\) trajectories.

The results are presented in Figure \ref{fixed budget}. It appears that the Stratonovich perturbation and the nonreversible one offer a significant computational advantage in this setting.

\begin{figure}[H]
\begin{minipage}{0.48\textwidth}
        \centering
    \includegraphics[width=0.8\linewidth]{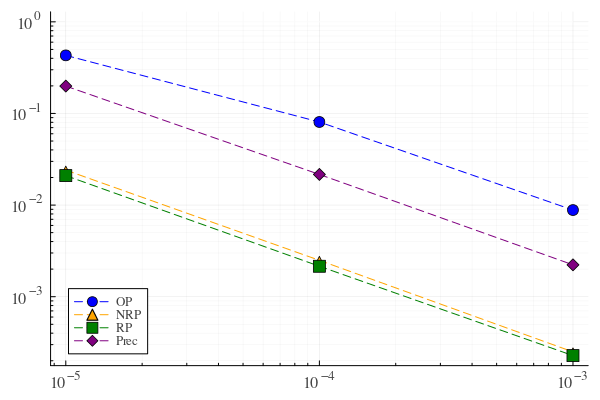}
    \caption{Relative MSE for different \(h\).}
    \label{fixed budget}
\end{minipage}
\hfill
\begin{minipage}{0.48\textwidth}
    \centering
    \includegraphics[width=0.8\linewidth]{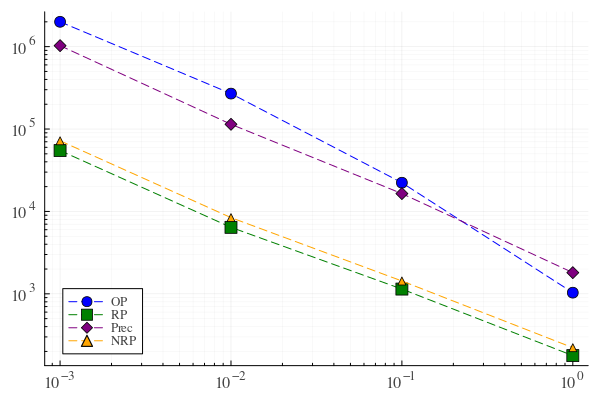}
    \caption{Number of step \(N\) to reach an error \(\epsilon\).}
    \label{fixed error}
\end{minipage}

\end{figure}

In the last experiment presented in Figure \ref{fixed error}, we use the splitting scheme \eqref{splitting stratonovich} and plot the number of step required to reach a fixed error. We simulate \(M=10^4\) trajectories of the dynamics in dimension \(d=20\) with a size step \(h = 10^{-4}\). It appears again that the Stratonovich perturbation behaves similarly to the nonreversible one, with a significant lead on the other process.
%
%
\section{Conclusions}\label{sec:conclusions}

The performance of the Langevin dynamics perturbed by a Stratonovich perturbation--first introduced in~\cite{StraoAbd19}--as a sampling scheme was studied in this paper. We showed that, for the case when the target distribution is Gaussian, an optimal Stratonovich perturbation with respect to maximizing the spectral gap exists and can be constructed in an algorithmic manner. Furthermore, we presented numerical experiments in which the superior performance of our proposed sampling scheme, relative to other Langevin-based sampling schemes, was demonstrated.

Several open questions remain. We mention here a few that are currently under investigation. First, the fully rigorous convergence analysis of the discretized Stratonovich-perturbed Langevin dynamics is an interesting question. Furthermore, the implementation of the proposed sampling scheme to high dimensional problems has not been explored in this paper. Finally, the rigorous analysis of the Stratonovich sampler for non-Gaussian target distributions remains open. Tackling this problem would require tools that are quite different from the ones that were used in the present article.

\section*{Acknowledgments}
G.L. and G.V. were partially supported by the Swiss National Science Foundation, projects No. 200020$\_$214819, No. 200020$\_$192129 and No. 10009199. G.A.P. is partially supported by an ERC-EPSRC Frontier Research Guarantee through Grant No. EP$/$X038645, ERC Advanced Grant No. 247031.

The AI Claude (Anthropic) was used during the preparation of this manuscript to chase typos in English phrasing and mathematical statements.

\bibliographystyle{abbrvnat}
\addcontentsline{toc}{section}{References}
\bibliography{biblio}

\begin{appendices}

\section{Complementary results for the 2 dimensional cases}\label{Complementary results for the 2 dimensional cases}

We recall that in Section \ref{Higher moments} we define
\[
G^{(n)}(k) \coloneqq \alpha_k^{(n)} - S_k^{(n)}, \quad k \in \{0, \dots, n\},
\]
where
\[
S_k^{(n)} = \frac{\gamma_{k-1}^{(n)}}{\eta_n(k-1)} + \gamma_k^{(n)} \eta_n(k),
\]
\[
\gamma_k^{(n)} = 2\sqrt{(n-k)(2n-2k-1)(k+1)(2k+1)},
\]
\[
\alpha_k^{(n)} = 4n + 8nk - 8k^2 - 2k.
\]

\begin{lemma}\label{borne n ge 11}
For all \(n \geq 1\) and all \(k = 0, \dots, n\), let
\[
\eta_n(k) = \left(1+\frac{1}{n}\right)\sqrt{\frac{2(n-k)}{2(n-k)-1}}.
\]
For all \(n \geq 11\) and all \(k = 0, \dots, n\), we have
\[
G^{(n)}(k) \geq \frac{5}{2}.
\]

\end{lemma}

\begin{proof}[Proof of lemma \ref{borne n ge 11}]
We first simplify \(S_k^{(n)}\). We obtain

\begin{align*}
\frac{\gamma_{k-1}^{(n)}}{\eta_n(k-1)}
&= \frac{1}{1+1/n}
\sqrt{2}(2n-2k+1)\sqrt{2k^2-k}, \\
\gamma_k^{(n)}\eta_n(k)
&= (1+1/n)\,2\sqrt{2}(n-k)\sqrt{2k^2+3k+1}.
\end{align*}

Thus,
\[
S_k^{(n)} =
\sqrt{2}\left[
\frac{1}{1+1/n}(2n-2k+1)\sqrt{2k^2-k}
+ (1+1/n)2(n-k)\sqrt{2k^2+3k+1}
\right].
\]
Using the inequalities for \(k \ge  1\)
\[
\sqrt{2k^2-k} \le \sqrt{2}k - \frac{\sqrt{2}}{4}, 
\quad
\sqrt{2k^2+3k+1} \le \sqrt{2}(k+1),
\]
we deduce
\begin{align*}
S_k^{(n)} 
&\le 2\left[
\frac{1}{1+1/n}(2n-2k+1)(k - 1/4)
+ (1+1/n)2(n-k)(k+1)
\right], \quad k \ge 1, \\
\quad S_0^{(n)} 
&= \sqrt{2}(1+1/n)2n.
\end{align*}

Now, we inject the upper bound in \(G\) for \(k\ge1\). 
\[
G^{(n)}(k) \ge  4n + 8nk - 8k^2 - 2k - 
2\left[ \frac{1}{1+1/n}(2n-2k+1)(k - 1/4) + (1+1/n)2(n-k)(k+1) \right].
\]
After simplification, we obtain
\[
G^{(n)}(k) \ge \frac{8 k + 8 k^2 - 8 n + 4 k n - 7 n^2 - 2 k n^2 + 2 n^3} {2 n (1 + n)}.
\]
The global minimum in \(k\) of the parabola \( 8 k^2 +  k (8 +4n - 2 n^2) +(- 8 n - 7 n^2  + 2 n^3)\) is reached at
\[
k^{\star} =\frac{n^2 - 2n -4}{8} \in [0,n] \iff  1 < n < 11.
\]
Since it is clear that \(G^{(n)}(0) = (4-\sqrt{8})n - 1\) is larger than the lower bound for all \(n\ge 11\), we conclude that, for \(n \geq 11\), the minimum is attained on the boundaries \(k = 1,n\). Therefore, we have for all \(n \geq 11\) and for all \(k = 0,\dots,n\)
\[
G^{(n)}(k)  \geq \min_{n \ge 11} \left\{ \frac{(16 - 4 n - 9 n^2 + 2 n^3)}{2n( 1 +  n)}, \frac{5n}{2n+2}  \right\} = \frac{5n}{2n+2}  \geq \frac{5}{2}.
\]

\end{proof}

\begin{lemma}\label{borne n leq 11}
For all \( 2 \leq n \leq 11\) and all \(k = 0, \dots, n\), let
\[
\eta_n(k) = \left(1+\frac{1}{10}\right)\sqrt{\frac{2(n-k)}{2(n-k)-1}}.
\]
For all \(1 \leq n \leq 11\) and all \(k = 0, \dots, n\), we have
\[
G^{(n)}(k) \geq \frac{9}{11}.
\]
\end{lemma}

\begin{proof}[Proof of lemma \ref{borne n leq 11}]
Substituting the explicit formulas for $\gamma_n(k)$ and $\eta_n(k)$, we can write
\begin{align*}
\frac{\gamma_n(k-1)}{\eta_n(k-1)} &= \frac{2\sqrt{(n-k+1)(2n-2k+1)k(2k-1)}}{\frac{11}{10}\sqrt{\frac{2(n-k+1)}{2n-2k+1}}} \\
&= \frac{10\sqrt{2}}{11}(2n-2k+1)\sqrt{k(2k-1)}.
\end{align*}
and
\begin{align*}
\gamma_n(k)\eta_n(k) &= 2\sqrt{(n-k)(2n-2k-1)(k+1)(2k+1)} \cdot \frac{11}{10}\sqrt{\frac{2(n-k)}{2n-2k-1}} \\
&= \frac{11\sqrt{2}}{5}(n-k)\sqrt{(k+1)(2k+1)}.
\end{align*}
After explicit substitution, the algebraic expression for $G^{(n)}(k)$ when $k \ge 1$ reduces to
\[
G^{(n)}(k) = \alpha_k^{(n)} - \frac{10\sqrt{2}}{11}(2n-2k+1)\sqrt{2k^2-k} - \frac{11\sqrt{2}}{5}(n-k)\sqrt{2k^2+3k+1}.
\]
We first check the endpoint of the domain. At $k = 0$
    \[
    G_n(0) = \alpha_n(0) - \gamma_n(0)\eta_n(0) = \left(4 - \frac{11\sqrt{2}}{5} \right)n.
    \]
For $k \ge 1$, we apply the following two algebraic inequalities to eliminate the square roots
\[
\sqrt{2k^2-k} \le \sqrt{2} \left(k - \frac{1}{4}\right), \qquad \sqrt{2k^2+3k+1} \le \sqrt{2}(k+1).
\]
Substituting these upper bounds into $G^{(n)}(k)$ produces a lower bound. We obtain
\begin{align*}
G^{(n)}(k) &\ge \alpha_k^{(n)} - \frac{20}{11}(2n-2k+1)\left(k - \frac{1}{4}\right) - \frac{22}{5}(n-k)(k+1) \\
&= (4n + 8nk - 8k^2 - 2k) - \frac{5}{11}(2n-2k+1)(4k-1) - \frac{22}{5}(n-k)(k+1).
\end{align*}
Expanding these expressions and placing them over the common denominator $55$ yields the clean quadratic polynomial $H_n(k)$
\[
G^{(n)}(k) \ge H_n(k) = \frac{2k^2 - 2kn - 18k + 28n + 25}{55}.
\]

We treat $H_n(k)$ as a continuous function of $k$ on the interval $[1, n]$. Because the leading coefficient of $k^2$ is $\frac{2}{55} > 0$, the parabola convex. The coordinates of its global minimum \(k^{\star}\) is
\[
 k^{\star} = \frac{n + 9}{2}.
\]
We evaluate the minimum of $H_n(k)$ across the domain $n \in [2, 11]$ by analyzing the position of this minimum.

\begin{enumerate}
    \item \textbf{Monotonic Regime ($2 \le n \le 8$):} For these values of $n$, the vertex lies strictly outside the right boundary of our domain ($ k^{\star} > n$). Consequently, $H_n(k)$ is strictly decreasing on $[1, n]$, meaning its minimum is localized precisely at the right endpoint $k = n$
    \[
    H_n(n) = \frac{2n^2 - 2n^2 - 18n + 28n + 25}{55} = \frac{2n + 5}{11}.
    \]
    Over the range $2 \le n \le 8$, this endpoint evaluation is minimized at the smallest index, $n = 2$
    \[
    H_n(n) \geq H_2(2) = \frac{9}{11}.
    \]

    \item \textbf{Remaining cases ($9 \le n \le 11$):} For $n = 9, 10, 11$, the vertex shifts inside the interval $[1, n]$. Evaluating the polynomial directly at $k^\star = \frac{n+9}{2}$ yields
    \[
    H_n\left(\frac{n+9}{2}\right) = \frac{-2n^2 + 56n -227}{55}.
    \]
    Evaluating this vertex function for $n=9, 10, 11$ produces the values $\frac{23}{11}, \frac{133}{55}, \frac{147}{55} $ respectively.
\end{enumerate}

Comparing the minima across all cases confirms that the global minimum of the lower bound is reached at $n=2, k=2$ with a value of \(\frac{9}{11}\).
\end{proof}

\end{appendices}

\end{document}